\documentclass[a4paper,10pt]{amsart}
\usepackage[utf8]{inputenc}
\usepackage{amsthm}
\usepackage[section]{placeins}
\usepackage{graphicx}
\usepackage{subcaption}
\usepackage{verbatim}
\usepackage{url}
\usepackage{lineno}
\usepackage{todonotes}
\nolinenumbers

\title{How does Noise Induce Order?}
\author{Isaia Nisoli}
\address{Universidade Federal do Rio de Janeiro - Instituto de Matemática - Av. Athos da Silveira Ramos, 149 - Edifício do Centro de Tecnologia, Bloco C (Térreo) - Cidade Universitária}
\address{Department of Mathematics, Hokkaido University, N10 W8, Kita-ku, Sapporo 001-0010, Japan}
\address{RIES, Hokkaido University, N20 W10, Kita-ku, Sapporo 001-0020, Japan}
\email{nisoli@im.ufrj.br, isaia.nisoli@es.hokudai.ac.jp}


\newcommand{\Var}{\textrm{Var}}
\newcommand{\Lip}{\textrm{Lip}}
\newcommand{\esssup}{\textrm{esssup}}

\newtheorem{theorem}{Theorem}[subsection]
\newtheorem{lemma}[theorem]{Lemma}
\newtheorem{conjecture}[theorem]{Conjecture}
\newtheorem{corollary}[theorem]{Corollary}
\newtheorem{definition}[theorem]{Definition}

\theoremstyle{remark}
\newtheorem{remark}[theorem]{Remark}

\begin{document}

\begin{abstract}
In this paper we present a general result with an easily checkable condition
that ensures a transition from chaotic regime to regular regime in random dynamical systems
with additive noise.
We show how this result applies to a prototypical family of 
nonuniformly expanding one dimensional dynamical systems,
showing the main mathematical phenomenon behind Noise-induced Order.

\smallskip
\noindent \keywordsname: Noise-induced order, Random dynamical systems, Non-uniformly hyperbolic,
Unimodal maps, Lyapunov exponents
\end{abstract}

\maketitle

\section{Introduction}
This article deals with the behavior of one dimensional nonuniformly 
hyperbolic systems with random additive noise.

A nonuniformly hyperbolic one dimensional dynamical system is a dynamical system in which
expansion and contraction coexists; the behavior of such a system 
is a delicate balance between how often the orbits of such a system 
visit the expanding and contracting region.
Such a system is called non uniformly expanding when the system visits more often 
the expanding region of the phase space.

Such balance may depend in non-trivial ways from a parameter: classical examples of unimodal 
maps, as the quadratic family, have a dense subset of parameters for which the deterministic
dynamic presents an attracting periodic orbit, called \textbf{regular} parameters and a 
positive Lebesgue measure Cantor set of parameters for which the dynamic shows chaotic behavior
called \textbf{stochastic} parameters. 

In this paper we will study a generalization of the quadratic family, allowing
the order of the critical point to vary, the family
\[
T_{\alpha, \beta}(x) = 1-2\beta |x|^{\alpha};
\]
these are symmetric unimodal maps, defined on $[-1,1]$, for $\alpha\in[2, +\infty)$, $\beta \in (0,1]$.

We will study the behavior under iterations of these maps with the addition
of a random noise at each iteration step, i.e.,
\[
X_{n+1} = T_{\alpha, \beta}(X_n)+\Omega_{\xi},
\]
where $\Omega_{\xi}$ is a random variable which takes values in $[-\xi, \xi]$
with density 
\[
\rho_{\xi}(x) = \frac{1}{\xi}\rho\left(\frac{x}{\xi}\right),
\]
where $\rho$ is a positive $BV$ density on $[-1,1]$.
We call $\xi$ the \textbf{amplitude} of the noise;
we denote the points of the orbits with a capital $X$ to stress the fact that they 
are random variables.
This is called a \textbf{random dynamical system with additive noise of amplitude $\xi$}.

We will show that when $\beta=1$ and $\alpha$ is bigger than a computable constant $\tilde{\alpha}>2.67835$\footnote{ the value of $\tilde{\alpha}$ 
is contained in $[2.67834, 2.67835]$, therefore, our result does 
not apply to the case $\alpha=2$, the quadratic family}
as the noise amplitude increases, the system transitions from a chaotic behavior
to an ordered behavior; this transition is measured quantitatively
by a transition of the Lyapunov exponent associated to the stationary measure
from positive to negative.

This surprising phenomenon is called in the literature Noise Induced Order and 
was first observed in numerical simulations of a model of the Belosouv-Zhabotinsky
reaction \cite{MT}, called the BZ map; a proof of its 
existence for the BZ map was given recently in  \cite{GaMoNi}.

In this paper we show the main mechanism behind this phenomenon;
the presence of noise changes the statistical properties of the dynamical system, in particular,
if we start with a non-uniformly expanding map, adding noise may break 
the delicate balance between expansion and contraction, and the average long term 
behavior changes from expanding to contracting.

Many Noise Induced phenomena \cite{Yu1, Yu2} are of strong interest for the  applied mathematical 
community and in general for applied sciences but until recently they have not woken the interest of 
the dynamical system community. Important results
have been reached in \cite{BY1, BY2}, in the study of the Henon and the Standard map with noise.

\begin{comment}
While for small noise amplitudes, the distribution of the orbits of the dynamical system
is similar to the distribution of the orbits of the deterministic system (a concept called
\textbf{stochastic stability}) for large noise the orbits of the random dynamical system 
distribute  themselves uniformly over the interval and, if $\alpha>\tilde{\alpha}$ visit 
more often the contracting part; in an intuitive way, the underlying deterministic dynamical 
system becomes hidden by the jumps given by the noise. 

Please remark that, under some hypothesis on the noise kernel, when the amplitude of the noise 
grows, the (periodic or reflecting) boundary conditions force the noise to be uniformly distributed
on $[-1, 1]$; therefore the stationary measure of the system at the large noise limit 
is uniformly distributed, which is the rationale behind item [R3] in Theorem \ref{thm:result}. 
\end{comment}

In figure \ref{fig:Lyap5} we have a plot of some numerical experiments on the family 
$T_{\alpha, \beta}$,
where fixed $\beta=1$, for each exponent $\alpha$ (in the vertical axis)
and noise amplitude $\xi$ (in the horizontal axis) we compute $200$ orbits of length $10000$, 
each one with a randomly chosen starting point and different random
realizations of the noise, and compute the average of $\ln(|T'(x)|)$ with 
respect to the length of the orbit, and take the average of these 
Birkhoff averages; the rationale behind this is that supposing that the simulated system 
satisfies some type of Central Limit Theorem
the mean of the finite time Birkhoff averages of all these orbits is a a better estimator of
the Lyapunov exponent than the average along an individual orbit.

This plot hints that Noise Induced Order may be present in the family $T_{\alpha, \beta}$;
on the left side of the plot, which presents the value of the estimator when the
noise amplitude is $0$ the estimator is positive.
On the right side of the plot, for noise amplitude $1.0$, we can see that if
the order of the critical point is big enough the estimator is negative.

More complex behavior can be conjectured from this plot: 
there are values of $\alpha$ for which we can observe multiple sign changes,
but the results in this paper only allow us to prove the existence of one transition.
The existence of multiple transitions could be proved by using Computer Aided Tools 
as the ones used in \cite{GaMoNi, ChiGaNiSa}.

\begin{figure}\label{fig:dynlyap}
\begin{subfigure}{.475\textwidth}
\includegraphics[width=60mm,keepaspectratio=true]{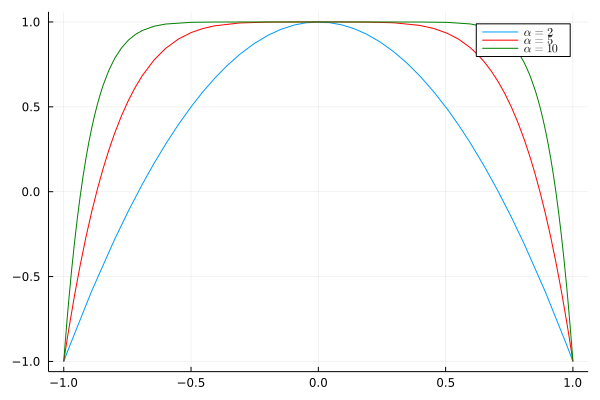}
     % Dynamic.png: 600x400 px, 72dpi, 21.17x14.11 cm, bb=0 0 600 400
\caption{$T_{\alpha, 1}$ for $\alpha=2, 5, 10$} \label{fig:dynamic}\
\end{subfigure}
\begin{subfigure}{.475\textwidth}
\includegraphics[width=60mm,keepaspectratio=true]{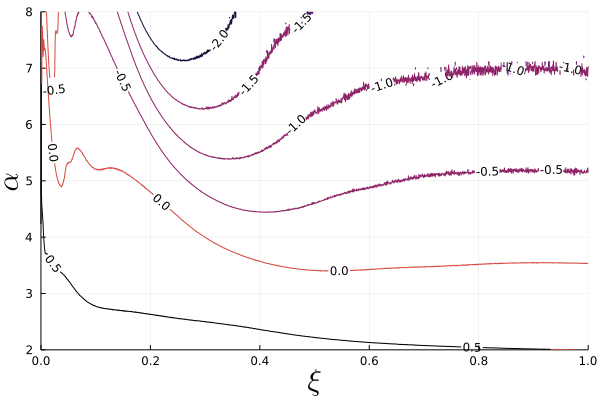}
     % numlyapalpha5.png: 600x400 px, 72dpi, 21.17x14.11 cm, bb=0 0 600 400
\caption{Lyapunov exponent in function of noise amplitude $\xi$ 
and order of the singularity $\alpha$ \label{fig:Lyap5}}
\end{subfigure}
\caption{The family $T_{\alpha, 1}$ and estimate of its Lyapunov exponent
as $\alpha$ and the noise size $\xi$ vary}
\end{figure}

The article \cite{MT} has been highly influential in the applied sciences; we think
that a sufficient condition for the existence of this phenomena is extremely interesting
both for the researchers in dynamical systems, since we present a wide family of examples 
whose deterministic behavior and behavior under the action of noise are different, 
and for the applied scientists, since this sufficient condition is easily checkable.

Our paper shows that Noise Induced Order is strongly linked with nonuniformly 
hyperbolic dynamics and that the existence of this kind of phenomena stresses 
the importance of the study of random dynamical systems beyond stochastic stability.
We think that our paper will contribute to show the richness of the behavior of 
dynamical systems with additive noise.

Our choice of the title is a direct answer to \cite{MT}; indeed, from the 
results of \cite{GaMoNi} and the results in the current paper, 
we can assert that we found the main mechanism for Noise Induced Order in the 
$1$-dimensional case.
Please remark that to apply the techniques in the present article to the BZ map 
from \cite{GaMoNi} we need a computer assisted step:
the BZ map does not fit in our framework to study stochastic stability for 
nonuniformly hyperbolic maps but our argument works once positive Lyapunov exponent and contraction of the space of average $0$ functions 
is proved for a small noise amplitudes, 
which is the difficult part of \cite{GaMoNi} and the main computer aided estimate.

\subsection{Statement of the results}

In this paper we prove that, under some assumptions, for all noise amplitudes
$\xi>\xi_0$ the random dynamical system 
has a unique ergodic absolutely continuous stationary measure $\mu_{\xi}$.

The Birkhoff Ergodic Theorem for Random Dynamical Systems tells us that, for a fixed 
noise size $\xi$, for $\mu_\xi$-a.e. initial condition $x_0$ and for almost all 
noise realizations we have that
\[
\lim_{n\to +\infty}\frac{1}{n}\sum_{i=0}^n \ln(|T'(X_i)|) = \int \ln(|T'|) d\mu_{\xi}.
\]

As $\xi$ varies we are interested in the behavior of the Lyapunov exponent
as a function of noise amplitude; remark that in the next formula, as $\xi$ varies 
$\mu_{\xi}$ is varying:
\[
\lambda(\xi) = \int \ln(|T'|) d\mu_{\xi};
\]
as in \cite{GaMoNi}, we define Noise Induced Order as follows.

\begin{definition}
We say that a system exhibits \textbf{Noise Induced Order} if there exist
$0<\xi_1<\xi_2$ such that for all $\xi\geq \xi_1$ the 
system has a unique stationary measure with density $f_{\xi}$ 
and the Lyapunov exponent of the stationary measure transitions continuously 
from positive to negative, i.e.,
$\lambda(\xi_1)>0$, $\lambda(\xi_2)<0$.
\end{definition}
\begin{remark}
There is an ongoing discussion in the community on the ``right'' definition 
of Noise Induced Order; in \cite{MT} are indicated:
\begin{itemize}
\item sharpening of power spectrum,
\item abrupt decrease of entropy,
\item appearance of negative Lyapunov number,
\item localization of orbit.
\end{itemize}
In \cite{GaMoNi} the existence of a transition for the Lyapunov exponent from positive
to negative was used as a definition of Noise Induced Order.
The continuity argument in the present paper shows that there exists a ``big'' noise amplitude
such that for all noise amplitudes bigger than this given noise amplitude, the Lyapunov exponent
is negative.
\end{remark}
\begin{remark}
The definition allows a deterministic map with negative Lyapunov exponent 
to show noise induced order: a regular parameter under the action of noise may 
show a transition to positive Lyapunov exponent for a small noise amplitude and 
a negative Lyapunov exponent for a larger noise amplitude.
\end{remark}

Our method to prove the existence of this transition is quite general and it follows from two simple
observations: the first one is that once contraction of the space of average $0$ functions in $BV$ (
also called exponential decay of correlations in $BV$) is proved for a noise
size $\xi_0$, the Lyapunov exponent is continuous with respect to $\xi$ for all $\xi>\xi_0$.

The second one is that as the noise amplitude grows, the density of the stationary measure becomes uniform,
and therefore, the limiting behavior of the Lyapunov exponent is the average of $\ln(|T'|)$
with respect to the uniform density on $[-1,1]$.

Many results on stochastic stability have been proved \cite{ArPaPi, AlAr03, AlVil13, Shen} 
that hint on the direction that positive Lyapunov exponent
may imply stochastic stability (see also the conjecture in \cite{ViaCo}).
Therefore we would like to state the following conjecture.
\begin{conjecture}\label{conj:NIO}
Let $T:I\to I$ be a piecewise $C^1$ dynamical system, nonsingular with respect to Lebesgue measure, 
which admits a unique 
absolutely continuous invariant measure with positive Lyapunov exponent; if
\[
-\infty<\int_{-1}^1 \ln(|T'|)\frac{dm}{2}<0
\]
then the associated random dynamical system with additive noise with bounded variation density
presents Noise Induced Order.
\end{conjecture}
\begin{remark}
The uniform density on $[-1,1]$ is $1/2$, which is the reason why many 
$1/2$ appears in the hypothesis above and in the conditions below.

While this notation is unneccessary, we would like to state the conditions 
in this form, to stress the mechanism underlying the transition.
\end{remark}

While we cannot prove this conjecture in its full generality, due to the technical difficulties 
involved in proving stochastic stability in a general setting, 
in this paper we prove the following theorem; please note that the hypothesis 
denoted by D are hypothesis on the deterministic system, while 
the hypothesis denoted by R are hypothesis on the associated random dynamical system
with additive noise.

\begin{theorem}\label{thm:result}
Let $T:[-1,1]\to [-1,1]$ be a piecewise $C^1$ nonsingular dynamical system such that
\begin{enumerate}
     \item[D1] admits a unique absolutely continuous invariant probability measure $\mu_0$ with density $f_0$,
     \item[D2] $\int_{-1}^1 \ln(|T'|)d\mu_0>0$,
     \item[D3] $\ln(|T'|) \in L^p$ for some $p>1$.
\end{enumerate}
Let $\mu_{\xi}$ be a fixed point for $L_{\xi}$, the annealed transfer operator (Defined in \ref{def:annealed}) and let
\[
\lambda(\xi) = \int_{-1}^1 \ln(|T'|)d\mu_{\xi}.
\]
Suppose now:
\begin{enumerate}
\item[R1] there exists $\xi_0$ such that $\lambda(\xi)$ is well-defined and continuous in $[0,\xi_0)$,
\item[R2] there exist $\xi_1$ in $[0,\xi_0)$, $C>0, \theta<1$ such that  $||P^n_{\xi_1}|_{\mathcal{U}_0}||_{BV}\leq C\theta^n$,
      where $P_{\xi_1}$ is the annealed Perron-Frobenius operator associated to 
      the random dynamical system with noise size $\xi_1$, defined in Definition \ref{def:annealed} and $\mathcal{U}_0$
      is the subspace of BV functions with average $0$ defined in Definition \ref{def:averagezero},
\item[R3] $-\infty <\int_{-1}^1 \ln(|T'(x)|)dm/2 < 0$,
\item[R4] the noise kernel is a mother noise kernel (Definition \ref{def:mothernoise}).
\end{enumerate}
Then, the function $\lambda(\xi)$ is well defined for and continuous for $\xi\geq 0$ and the map $T$ exhibits Noise Induced Order.
\end{theorem}

The proof of the theorem is found in Section \ref{sec:suff}.

We will try to give an intuition behind the hypothesis for this Theorem and our proof.
Hypothesis D1 and D2 are telling us that the original system has positive Lyapunov exponent; 
the deterministic system is chaotic, D3 is a mild regularity assumption.
Hypothesis R1 follows from the stochastic stability and tells us that Lyapunov exponent is continuous in a small neighborhood of $0$. 
Hypothesis R2 kickstarts our continuity argument; remark that even if the underlying dynamical 
system has subexponential decay of correlations, this hypothesis may be satisfied, due to the 
smoothing properties of noise.
Hypothesis R3 is an hypothesis on the behavior of the system when the noise is big; as the noise 
size increases the noise moves the random orbits uniformly inside the system, so 
the Lyapunov exponent along random orbits is negative.
Hypothesis R4, the fact that the noise density is a mother noise kernel can be intuitively
understood as the fact that, as the noise amplitude increases, the support of the noise 
contains the support of smaller amplitude noises.

The consequences for the family $T_{\alpha, \beta}$ are summarized in the following theorem
and are proved in section \ref{sec:consequences}
\begin{theorem}\label{thm:consequences}
Let $T:[-1,1]\to [-1,1]$ be of the form
\[
T_{\alpha,\beta}(x)=1-2\beta |x|^{\alpha}.
\]
For all $\alpha>\tilde{\alpha}>2.67835$ there exists an $\epsilon(\alpha)$ such that for all 
$\beta \in (1-\epsilon(\alpha), 1]$ the map $T_{\alpha, \beta}$ exhibits
Noise Induced Order.
\end{theorem}

\subsection{Structure of the paper}
We start in Section \ref{sec:annealed} where we introduce the annealed transfer operator,
prove that contraction of the space of average $0$ function for a noise amplitude implies
contraction for all bigger noise amplitudes and prove that the Birkhoff averages
of $L^1(m)$ observables are continuous with respect to noise size and the underlying dynamics 
 once contraction of average $0$ functions is established.
Section \ref{sec:suff} is a small section, where we present the proof of Theorem \ref{thm:result}.
Section \ref{sec:consequences} studies the family $T_{\alpha,\beta}$, by producing results
on its stochastic parameters, stochastic stability and showing the conditions on
$\alpha$ and $\beta$ that imply that maps in the family
present noise induced order.

\section{The annealed transfer operator}\label{sec:annealed}

\subsection{Generalities on the involved functional spaces}
\begin{definition} 
Let $[a,b]\subset \mathbb{R}$ be an interval endowed with the Lebesgue measure $m$;
we denote by $L^r([a,b])$ the Banach space of real valued functions such that
\[
||f||_{L^r([a,b])}:=\sqrt[r]{\int_a^b |f|^r dm}<+\infty.
\]
We will drop the interval from the notation when clear from the context.
Of particular interest for us is $L^1([-1,1])$.
     
We define $L^{\infty}([a,b])$ to be the Banach space of real valued functions such that
\[
||f||_{L^\infty([a,b])}:=\esssup_{x\in [a,b]}|f(x)|<+\infty, 
\]
where the essential supremum is the smallest real number $a$ such that $|f(x)|\leq a$
for $m$-almost every $x$ in $[a,b]$.
\end{definition}
\begin{definition}\label{def:density}
We will call a \textbf{density} a nonnegative function $f\in L^1([a,b])$ such that
\[
\int_a^b f dm = 1 
\]
\end{definition}
\begin{lemma}
If $f$ is a density, then 
\[
||f||_{L^1([a,b])} = 1.     
\]
\end{lemma}
\begin{proof}
This follows from the definition, since $f$ is nonnegative
\[
\int_a^b |f| dm = \int_a^b f dm = 1.
\]
\end{proof}
\begin{definition}
Let $\phi:[a,b]\to \mathbb{R}$ be a real valued function on $[a,b]$, we define the \textbf{variation} of $\phi$ on $[a,b]$
as
\[
\Var_{[a,b]}(\phi) = \sup_{\mathcal{P}}\sum_i |\phi(x_{i+1})-\phi(x_i)|     
\]
where $\mathcal{P}$ is any partition of $[a,b]$ with endpoints $x_i$.
If the variation of $\phi$ is finite, we say that $\phi$ is a function of bounded variation on 
$[a,b]$.
The functions of bounded variation on $[a,b]$ are a Banach space when equipped with the norm
\[
||\phi||_{BV([a,b])}:=||\phi||_{L^1([a,b])}+\Var_{[a,b]}(\phi).     
\]
When the domain $[a,b]$ is clear, we will drop the subscript.
\end{definition}
     
In the following lemma, we need to pay some attention on the domain of definitions of 
the functions and to the support of measures:
the convolution is defined in general for functions defined on the real line, while 
we speak of functions which are $L^1$ or bounded variation on intervals.
     
\begin{definition}
Denote by $\chi_X$ is the characteristic function of the set $X$.
     
In the following we will denote by 
\[
\hat{\phi} = \phi\cdot \chi_{[-\xi, \xi]}    
\]
and by    
\[
\hat{f} = f \cdot \chi_{[-1,1]}
\]
the functions that extend by $0$ outside their 
intervals of definition the functions $\phi$ and $f$ respectively.
     
Given a probability measure $\mu$ on $[-1, 1]$ we define its extension 
$\hat{\mu}$ on $\mathbb{R}$ as the unique measure $\hat{\mu}$ on $\mathbb{R}$
such that 
\[
\hat{\mu}(A) = \mu(A\cap [-1, 1])     
\]    
for all $A$ measurable in $\mathbb{R}$.
\end{definition}
\begin{lemma}\label{lem:extension}
The following are true:
\begin{enumerate}
     \item \label{it:extvar}$\Var_{[a,b]}(\hat{\phi}) \leq \Var_{[-\xi, \xi]}(\phi)+2 \sup_{[-\xi, \xi]}|\phi(x)|\leq 3 ||\phi||_{BV([-\xi,\xi])}$ for all interval $[a,b]$ that contains $[-\xi, \xi]$,
     \item \label{it:L1} $||\hat{f}||_{L^1([c,d])} = ||f||_{L^1([-1,1])}$ for all interval $[c,d]$ that contains $[-1, 1]$, 
     \item \label{it:prob} $\hat{\mu}([-1-\xi, 1+\xi])=1$ for all $\xi$.
\end{enumerate}
\end{lemma}
\begin{proof}
Both follow from the respective definitions.
We first prove item \ref{it:extvar}; let $\{x_1=a,\ldots,x_n=b\}$ be a partition of $[a,b]$,
without loss of generality we can suppose that $x_l=-\xi$, $x_{l+k}=\xi$.
     
Then 
\begin{align*}
&\sum_{i=1}^n |\hat{\phi}(x_{i+1})-\hat{\phi}(x_i)| = |\hat{\phi}(x_{l})|+\sum_{i=l}^{l+k-1} |\hat{\phi}(x_{i+1})-\hat{\phi}(x_i)|+|\hat{\phi}(x_{l+k})|\\
&= |\phi(x_{l})|+\sum_{i=l}^{l+k-1} |\phi(x_{i+1})-\phi(x_i)|+|\phi(x_{l+k})|\leq \Var_{[-\xi, \xi]}(\phi)+2 \sup_{[-\xi, \xi]}|\phi(x)|.
\end{align*}
     
We prove now item \ref{it:L1}:
\begin{align*}
\int_c^d |\hat{f}| dm = \int_c^d |f\cdot \chi_{[-1, 1]}| = \int_{-1}^1 |f|dm = ||f||_{L^1([-1,1])}. 
\end{align*}
     
Finally item \ref{it:prob}
\[
\hat{\mu}([-1-\xi, 1+\xi]) = \mu([-1-\xi, 1+\xi]\cap [-1, 1]) = \mu([-1,1]) = 1.
\]
\end{proof}

\subsection{Regularization properties of convolution on measures}
In this subsection we define what is the convolution of a measure with 
respect to a bounded variation function and prove some regularization properties of this operator;
the most important is that convolution of a measure with a bounded variation functions is 
a measure which is absolutely continuous with respect to Lebesgue.

\begin{definition}\label{def:convolutionmeas}
Let $\mu$ be any probability measure in $[-1,1]$, and let $\rho$ be a bounded variation 
function on $[-\xi, \xi]$ with $\int_{-\xi}^{\xi} \rho=1$; their convolution is the unique probability
measure $\hat{\rho}\ast \hat{\mu}$ on $\mathbb{R}$ such
that 
\[
\hat{\rho}\ast \hat{\mu}(A) = \int_{[-\xi, \xi]}\hat{\rho}(y)\hat{\mu}(A-y)dm(y), 
\]
where $A-y$ to denote the set $\{x-y \mid x\in A\}$.
\end{definition}

\begin{lemma}\label{lemma:propconvmeas}
The following properties of $\hat{\rho}\ast \hat{\mu}$ are true:
\begin{enumerate}
\item \label{it:prob2} $\hat{\rho}\ast \hat{\mu}([-1-\xi, 1+\xi]) = 1$,
\item \label{it:delta} if $\mu=\delta_p$, the Dirac-$\delta$ measure at $p\in [-1,1]$ 
we have that 
\[
\hat{\rho}\ast \hat{\delta_p} = \hat{\rho}(x-p)\cdot m(x);
\]
in particular, $\hat{\rho}\ast \hat{\delta_p}$ is absolutely continuous with respect to Lebesgue,
\item \label{it:acconv} if $\mu = f dm$ then $\hat{\rho}\ast \hat{\mu}$ has density $\hat{\rho}\ast \hat{f}$.
\end{enumerate}
\end{lemma}
\begin{proof}
Item \ref{it:prob2} follows from the definition and Item \ref{it:prob} in Lemma \ref{lem:extension}: 
\begin{align*}
\hat{\rho}\ast \hat{\mu}([-1, 1]) &= \int_{[-\xi,\xi]}\hat{\rho}(y)\hat{\mu}([-1-\xi, 1+\xi])dm(y)\\
&= \mu([-1, 1])\int_{[-\xi,\xi]}\hat{\rho}(y)dm(y) = 1.
\end{align*}
          
Item \ref{it:delta} follows from the definition, recalling that $\delta_p(A)=\chi_A(p)$: 
\begin{align*}
\hat{\rho}\ast \hat{\delta_p}(A) &= \int_{[-\xi, \xi]}\hat{\rho}(y)\hat{\delta_p}(A-y)dm(y)=\int_{[-\xi, \xi]}\hat{\rho}(y)\chi_{A-y}(p)\chi_{[-1,1]}(p)dm(y)\\
&=\int_{[-1-\xi, 1+\xi]}\chi_{[-\xi, \xi]}(y)\rho(y)\chi_{A}(p+y)dm(y)=\\
&=\int_{[-1-\xi, 1+\xi]}\chi_{[-\xi, \xi]}(z-p)\rho(z-p)\chi_{A}(z)dm(x)=\int_A \hat{\rho}(z-p)dm(z),
\end{align*}
where in the last line we used the change of variables $z=p+y$.
          
Item \ref{it:acconv} follows from the definition.
\end{proof}
     
We prove now a general result on sequences of absolutely continuous probability measures 
with uniformly bounded densities.
\begin{lemma}\label{lemma:seqboundeddensities}
Let $\mu_n = f_n dm$, $n\in \mathbb{N}$ be a sequence of absolutely continuous probability measures such that:
\begin{itemize}
     \item $f_n\in L^{\infty}(m)$ for all $n$
     \item there exists an interval $[a,b]$ such that $\mu_n(\mathbb{R}\setminus [a,b])=0$, for all $n$,
     \item $\mu_n$ converges weakly to $\mu$,
     \item there exists $M>0$ such that $||f_n||_{\infty}\leq M$ for all $x\in [a,b], n \in \mathbb{N}$.
\end{itemize}
then $\mu$ is an absolutely continuous probability measure, with $\mu(\mathbb{R}\setminus [a,b])=0$.
\end{lemma}
\begin{proof}
     
By Portmanteau theorem, weak convergence of $\mu_n$ to $\mu$ implies that 
for all open sets $A$
\[
\mu(A)\leq \liminf \mu_n(A).      
\]
This implies that 
\[
\mu(\mathbb{R}\setminus [a,b])\leq \liminf \mu_n(\mathbb{R}\setminus [a,b])=0.
\]
     
The fact that $\mu$ is a probability measure follows from definition of weak convergence.
     
Suppose now $B$ is a measurable set; we claim that if $m$
is the Lebesgue measure $m(B)=0$ implies $\mu(B)=0$.
     
Let $B$ measurable, without loss of generality we can suppose $B\subseteq [a,b]$,  
and let $A$ be any open set containing $B$; by weak convergence, Portmanteau theorem
and the fact that $||f_n||_{\infty}\leq M$ for all $n$ we have that:
\begin{align*}
\mu(B)\leq \mu(A)\leq \liminf \mu_n(A) = \liminf \int_A f_n dm \leq M \cdot m(A).
\end{align*}
     
We recall that the Lebesgue measure $m$ on $\mathbb{R}$ is outer regular, i.e., 
for all measurable sets $B$ we have that $m(B) = \inf \{ m(A) \mid \textrm{$A$ open}, B\subseteq A\}$;
taking the $\inf$ over all open sets $A$ containing $B$ on the right side of the inequality above 
implies that $\mu(B)\leq M\cdot m(B)$ and absolute continuity of $\mu$.
\end{proof}
\begin{remark}
The hypothesis that the $f_n$ are uniformly bounded is fundamental in the proof above,
and the theorem is false if it is not satisfied. An example is the sequence $f_n = 1/(2\xi) \chi_{[-\xi, \xi]}\cdot m$
which converges weakly to $\delta_0$.
\end{remark}

\begin{remark}
This Lemma is a folklore result in measure theory 
\cite{SE}; in the provided link, different proofs and a discussion of the result are provided. 

We sketch another proof, found at the provided link, with a more functional analytic approach:
by classical results $C^{\infty}([a,b])$ is dense in $L^2([a,b])$.
We define a sequence of functionals $T_n(g): C^{\infty}([a,b])\to \mathbb{R}$
by $T_n(g) = \int g d\mu_n$; by weak convergence, for each $g\in C^{\infty}([a,b])$ we can define 
$T(g):= \lim_{n\to +\infty} \int g d\mu_n = \int g d\mu$.
 
We show now that, since the $f_n$ are uniformly bounded $T$ can be extended to a functional $T:L^2([a,b])\to \mathbb{R}$;
this follows from the Cauchy-Schwarz inequality, since 
\[
|T_n(g)| = |\int g f_n dm| \leq ||g||_{L^2} ||f_n||_{L^2}\leq ||g||_{L^2} \sqrt{b-a}\cdot  M,     
\]
since the bound is uniform in $n$ the functional $T(g)$ can be extended to a bounded linear functional on $L^2$.
By Riesz representation theorem, there exists an $f\in L^2([a, b])$ such that 
\[
T(g)=\int g \cdot f dm,      
\]
which implies that $\mu = f \cdot m$ is absolutely continuos.
\end{remark}

\begin{lemma}\label{lemma:weakconv}
Let $\mu_n$ be a sequence of probability measures on $[-1,1]$ weakly converging to $\mu$.
Then $\hat{\rho}\ast \hat{\mu_n}$ converges weakly to $\hat{\rho}\ast \hat{\mu}$.
\end{lemma}
\begin{proof}
By definition of weak convergence we have that for all $\phi$ Lipschitz
on $[-1,1]$ we have that 
\[
\lim_{n\to+\infty} \int_{-1}^1 \phi d\mu_n = \int_{-1}^1 \phi d\mu.   
\]

Let $\phi$ be a Lipschitz continuous function on $[-1-\xi, 1+\xi]$, 
then, since $\int_{-\xi}^{\xi}\rho dm =1$ we have that 
\begin{align*}
&\left|\int \hat{\rho}(x+h-y)\phi(y)dm(y)-\int \hat{\rho}(x-y)\phi(y)dm(y)\right|\\
&=\left|\int \hat{\rho}(z)\left(\phi(x+h-z)-\phi(x-z)\right)dm(z)\right|\leq L \cdot h     
\end{align*}
where $L$ is the Lipschitz constant of $\phi$.

Now, for each $\phi$ Lipschitz on $[-1-\xi,1+\xi]$ we have
\[
\int \phi(x) \int \hat{\rho}(x-y) d\hat{\mu}_n (y) dm(x) = \int \int \phi(x) \hat{\rho}(x-y) dm(x)d\hat{\mu}_n(y),     
\] 
by the inequality above $\phi\ast \hat{\rho}$ is Lipschitz continuous and so 
is its restriction to $[-1, 1]$; therefore, for each $\phi$ Lipschitz on $[-1-\xi, 1+\xi]$
we have
\begin{align*}
&\lim_{n\to+\infty}\int \phi(x) d (\hat{\rho}\ast \hat{\mu}_n)(x) = \lim_{n\to+\infty}\int_{-1}^1 (\hat{\rho}\ast \phi)(x) d\mu_n(x)\\
&= \int_{-1}^1 (\hat{\rho}\ast \phi)(x) d\mu(x) = \int_{-1-\xi}^{1+\xi} \phi(x) d (\hat{\rho}\ast \hat{\mu})(x). 
\end{align*}
\end{proof}

We prove now the final result of this section, that shows
that convolution with a bounded variation kernel maps probability measures into
probability measures which are absolutely  continuous with respect to Lebesgue.

\begin{lemma}\label{lemma:measuretoac}
Let $\mu$ be a probability measure in $[-1,1]$, then $\hat{\rho}\ast \hat{\mu}$
is a probability measure on $[-1-\xi, 1+\xi]$, absolutely continuous with respect to Lebesgue. 
\end{lemma}
\begin{proof}
Recall from Definition \ref{def:convolutionmeas} that $\rho$ is a bounded variation 
function on $[-\xi, \xi]$ with $\int_{-\xi}^{\xi} \rho dm=1$.

The proof follows from Lemma \ref{lemma:seqboundeddensities};
let $\{ -1-\xi = x_0, \ldots x_{n+1} = 1+\xi \}$ be a partition of $[-1-\xi, 1+\xi]$ such that  
$x_{i+1}-x_i\leq 2/n$ for all $i = 0,\ldots, n$.
Let 
\[
\mu_n = \sum_{i=0}^n \mu([x_i, x_{i+1}]) \delta_{p_i}     
\]
where $p_i= (x_{i+1}+x_i)/2$ and $\delta_{p_i}$ is the Dirac-$\delta$ measure centered at $p_i$. Then $\mu_n$ converges weakly to $\mu$,
and $\hat{\rho}\ast \hat{\mu}_n$ converges weakly to $\hat{\rho}\ast \hat{\mu}$
by Lemma \ref{lemma:weakconv}.
          
By Lemma \ref{lemma:propconvmeas} Item \ref{it:delta} and linearity of convolution we have that 
\[
\hat{\rho}\ast \hat{\mu}_n = \sum_{i=0}^n \mu([x_i, x_{i+1}])\hat{\rho}(x-p_i) \cdot m, 
\]
which, for each $n$, is an absolutely continuous probability measure whose density is uniformly bounded,
i.e.,  
\[
\left| \sum_{i=0}^n \mu([x_i, x_{i+1}])\hat{\rho}(x-p_i) \right|\leq ||\rho||_{L^{\infty}([-\xi, \xi])}\sum_{i=0}^n \mu([x_i, x_{i+1}])\leq ||\rho||_{BV}.
\]
Then, by Lemma \ref{lemma:seqboundeddensities} we have that $\hat{\rho}\ast \hat{\mu}$ is an absolutely continuous
probability measure.

\end{proof}

\subsection{Regularization properties of convolution on densities}

\begin{lemma}\label{lem:reg}
Let $f\in L^1([-1, 1])$ and let $\phi$ be a bounded variation function on $[-\xi, \xi]$;
then, their convolution
\[
\hat{\phi}\ast \hat{f} (x):=\int_{-\infty}^{\infty} \hat{\phi}(x-y)\hat{f}(y)dy
\]
is a bounded variation function with support in $[-1-\xi, 1+\xi]$, such that
\[
\Var_{[-1-\xi, 1+\xi]}(\hat{\phi}\ast \hat{f})\leq \left(\Var_{[-\xi, \xi]}(\phi)+2 \sup_{[-\xi, \xi]}|\phi(x)|\right)||f||_{L^1([-1,1])}.    
\]

Morever, if $\phi(x)\geq 0$ and $\int_{[-\xi, \xi]} \phi(x)dm(x)=1$, then $||\hat{\phi}\ast \hat{f}||_{L^1([-1-\xi, 1+\xi])}\leq ||f||_{L^1([-1,1])}$.
\end{lemma}
\begin{proof}
Let $\tau_x$ be the translation operator on functions, i.e.,
$(\tau_y \hat{\phi})(x)=\hat{\phi}(x-y)$.
By definition
\[
\Var_{[-\xi,\xi]}(\hat{\phi}) = \Var_{[y-\xi,y+\xi]}(\tau_y\hat{\phi}).    
\]

We first remark that by definition of $\hat{\phi}$ and $\hat{f}$, their convolution
$\hat{\phi}\ast\hat{f}$ is $0$ outside $[-1-\xi, 1+\xi]$.

We observe now that for any partition $\mathcal{P}$ of $[-1-\xi, 1+\xi]$ we have that
\begin{align*}
\int_{-1}^1 \sum_i |\hat{\phi}(x_i-y)-\hat{\phi}(x_{i+1}-y)||\hat{f}(y)|dy \leq \int \Var_{[y-\xi, y+\xi]}(\tau_y \hat{\phi})|\hat{f}(y)|dy; 
\end{align*}
observing that $\tau_y \hat{\phi}$ is $0$ outside of $[y-\xi, y+\xi]$,
and that 
\[
\sum_i |\hat{\phi}(x_i-y)-\hat{\phi}(x_{i+1}-y)|\leq \Var_{[y-\xi, y+\xi]}(\tau_y \hat{\phi})\leq \Var_{[\xi, \xi]}(\phi)+2 \sup_{[-\xi, \xi]}|\phi(x)|,
\]
by definition of variation.

Therefore
\begin{equation}\label{eq:boundvar}
\int \sum_i |\hat{\phi}(x_i-y)-\hat{\phi}(x_{i+1}-y)||\hat{f}(y)|dy \leq \Var_{[-\xi, \xi]}(\hat{\phi})||f||_{L^1([-1,1])}
\end{equation}
remark that the fact that the left handside above is bounded will allow us to interchange the summation
and integral sign by Fubini-Tonelli theorem, and that on the right hand side we have the variation of $\phi$ and
the $L^1$ norm of $f$, by Lemma \ref{lem:extension}.

For any partition $\mathcal{P}$ of $[-1-\xi, 1+\xi]$ we have that
\[
\sum_i |\int \hat{\phi}(x_i-y)-\hat{\phi}(x_{i+1}-y)\hat{f}(y)|dy\leq \sum_i \int |\hat{\phi}(x_i-y)-\hat{\phi}(x_{i+1}-y)||\hat{f}(y)|dy 
\] 
exchanging the summation and integral sign and using \eqref{eq:boundvar} we obtain the thesis.

Suppose now $\int_{[-\xi, \xi]} \phi(x)dm(x)=1$, by the argument above we know 
that the convolution integral is bounded so we can exchange the order of integration;
remembering that $\hat{f}$ extends $f$ by $0$ outside $[-1,1]$ we have then:
\begin{align*}
\int_{-1-\xi}^{1+\xi}&\left|\int_{-\xi}^{\xi}\hat{\phi}(y)\hat{f}(x-y) dm(y) \right| dm(x)\leq \int_{-\xi}^{\xi} \hat{\phi}(y) \int_{-1-\xi}^{1+\xi} |\hat{f}(x-y)| dm(x) dm(y)\\
&= \int_{-\xi}^{\xi} \hat{\phi}(y) ||f||_{L^1([-1,1])} dm(y) = ||f||_{L^1([-1,1])}.         
\end{align*}
\end{proof}

\begin{remark}\label{rem:equivBV}
A useful characterization of bounded variation functions is the following 
approximation by smooth functions result, \cite[Theorem 3.9]{AmFuPa}.
A function $u\in L^1([a,b])$ is of bounded variation
if and only if there exists a sequence $u_n$ in $C^{\infty}([a,b])$ converging to $u$
in $L^1([a,b])$ and such that 
\[
\lim_{n\to+\infty}\int_{a}^b |u'_n|dm \leq V < +\infty      
\]
The smallest possible constant $V$ is the variation of $u$.
All of the proofs about regularity in our paper can be redone by using this characterization.
\end{remark}

\subsection{Definition of the annealed transfer operator}
\begin{definition}\label{def:transferoperator}
Let $T:[-1, 1]\to [-1,1]$ be a measurable map. The map $T$ induces an operator 
on $L:\mathcal{SM}([-1,1])\to \mathcal{SM}([-1,1])$ where $\mathcal{SM}([-1,1])$
is the space of signed measures on $[-1,1]$, defined in the following way:
if $\mu\in \mathcal{SM}([-1,1])$ then 
\[
     L\mu(A) = \mu(T^{-1}A)
\]
for all measurable sets $A$.
This operator is called the pushforward operator associated to $T$ or the \textbf{transfer operator}
associated to $T$.

The space of Lebesgue absolutely continuous measures is a vector subspace of $\mathcal{SM}([-1,1])$;
if $T$ is non-singular with respect to Lebesgue then $L$ preserves this subspace of absolutely continuous measures
and induces an operator from $L^1([-1,1])$ into itself called the 
\textbf{Perron-Frobenius} operator.
We will denote by $P$ the Perron-Frobenius operator.
\end{definition}
\begin{remark}
Given an absolutely continuous probability measure $\mu = f\cdot m$, with density $f$,
$Pf$ is the Radon-Nikodym derivative of $L \mu$ with respect to $m$ \cite{Sarig}.
\end{remark}

\begin{remark}
By definition, for any measurable function $\phi$, the pushforward operator satisfies the following duality 
formula
\[
\int_{-1}^1 \phi d(L\mu) = \int_{-1}^1 \phi\circ T d\mu.     
\]
\end{remark}
The following is a collection of basic properties of the Perron-Frobenius operator 
$P$, that are proved in the first pages of \cite{Sarig}, whose proof we omit.
\begin{lemma}[\cite{Sarig}]
The following statements are true.
\begin{enumerate}
     \item $Pf$ is the unique function in $L^1([-1,1])$ such that for all test function in $L^{\infty}(m)$:
     \[\int_{-1}^1 \phi \cdot Pf dm = \int_{-1}^1 \phi\circ T\cdot  f dm,\]
     \item $P$ is a positive linear operator, and $||P||_{L^1([-1,1])}=1$, 
     \item if $f$ is a density, then $Pf$ is a density.
\end{enumerate}
\end{lemma}

\begin{definition}
We will call \textbf{boundary condition} one of the two following maps:
\begin{itemize}
\item $\pi_P(x) = x$ mod $2$, called a \textbf{periodic boundary conditions},
\item $\pi_R(x) = (\min_{i\in \mathbb{Z}} |(x+1)-4i|)-1$, called a \textbf{reflecting boundary conditions}.
\end{itemize}
When the choice of the boundary condition is unimportant we will denote 
a boundary condition by $\pi$.
We will denote by $\pi_*$ the push-forward map acting on measures by 
\[
(\pi_*\mu)(A) = \mu(\pi^{-1}(A)).
\]
\end{definition}
\begin{remark}
In the definition of $\pi_P$ above we choose as representatives of the equivalence relation 
classes the points in $(-1,1]$.
\end{remark}
\begin{remark}\label{rem:abuse}
By abuse of notation $\pi_*$ will denote also the map that 
$\pi_*$  it induces on densities, i.e,
if $\mu$ has density $g$, then $\pi_*(g)$ is the density of 
$\pi_*\mu$; refer to Lemma \ref{lem:proj} for the conditions under
which this map is well defined and their proof.
\end{remark}
\begin{remark}
The map $\pi_*$ is well defined only on measures $\mu$ on $\mathbb{R}$
such that there exists an interval $[a,b]$ such that 
\[
\mu(\mathbb{R}\setminus [a,b]) = 0;     
\] 
by Lemma \ref{lemma:measuretoac} this is true for all the measures 
$\hat{\rho}\ast \hat{\mu}$ in our treatment.
\end{remark}

\begin{remark}\label{rem:extbounded}
Let $\pi^*$ be the map that associates to any  
$\phi$ bounded and measurable on $[-1,1]$ its extension 
$\hat{\phi}$ such that 
\[
\hat{\phi}(x) = \phi(\pi(x)),
\]
for a boundary condition $\pi$.

If $\mu$ is a measure on $\mathbb{R}$
such that there exists an interval $[a,b]$ such that 
\[
\mu(\mathbb{R}\setminus [a,b]) = 0;     
\] 
we have that 
\[
\int \phi d \pi_* \mu = \int \pi^*(\phi) d\mu.    
\]
\end{remark}

\begin{lemma}\label{lem:proj}
Let $\mu=f \cdot m$ be an absolutely continuous probability measure on $\mathbb{R}$, 
with density $f$ such that $f\equiv 0$ in $\mathbb{R}\setminus [a,b]$.
For any boundary condition $\pi$, $\pi_*(\mu)$ is an absolutely continuous 
probability measure on $[-1,1]$.

Moreover if $f$ is of bounded variation, then $\pi_* \mu$ has a bounded variation density.
\end{lemma}
\begin{proof}
Let $\pi_i$ be the restriction of $\pi$ to the interval $I_i = [-1+2i, 1+2i]$; by definition,
$\pi_i$ is one to one and affine.
Let $g$ be the density of $\mu$ and $g_i$ its restriction to $I_i$, then $\pi_* \mu$
has density $\tilde{g}:=\sum_i g_i(\pi_i^{-1}(x))$, where this sum is well defined since $g$ has bounded support.
Then
\[
||\tilde{g}||_{L^1([-1,1])}\leq \sum_{i} ||g_i||_{L^1(I_i)}=||g||_{L^1([a,b])}.
\]

If $g$ is of bounded variation, then:
\[
\Var_{[-1,1]}(\tilde{g})\leq \sum_{i} \Var_{I_i}(g_i) \leq \Var_{[a,b]}(g).     
\]
\end{proof}

\begin{definition}\label{def:mothernoise}
Let $\rho$ a bounded variation function such that
$\rho(x)\geq c> 0$ for all $x\in [-1,1]$, $\rho(x)=0$ outside $[-1,1]$ and 
$\int_{-1}^1 \rho(x)dm = 1$; we will call such a function
a \textbf{mother noise kernel}.

In the following, define
\[
\rho_{\xi}(x) := \frac{1}{\xi}\rho\left(\frac{x}{\xi}\right).    
\]
We will call $\xi$ the \textbf{amplitude of the noise}.
\end{definition}    
\begin{definition}
Let $T: [-1, 1]\to [-1,1]$ be a measurable non-singular function; a random dynamical 
system with noise amplitude $\xi$ with initial condition $x_0$ is 
a sequence of random variables  
\[
X_0 = x_0, \quad X_{n+1} = \pi(T(X_n)+\Omega_{\xi})     
\]
where $\Omega_{\xi}$ is a random variable with probability density $\rho_{\xi}$ and 
$\pi$ is either a periodic or reflecting boundary condition.
\end{definition}

\begin{definition}\label{def:annealed}
The annealed transfer operator $L_{\xi}$ associated to the system with noise is defined by 
\[
L_{\xi}\mu = \pi_* (\hat{\rho}_{\xi}\ast \widehat{L\mu})
\] 
where $\pi_*$ can be either periodic or reflecting boundary conditions.
\end{definition}
\begin{lemma}
The operator $L_{\xi}$ induces an operator $P_{\xi}$ acting on densities 
such that
\[
P_{\xi}f = \pi_*(\hat{\rho}_{\xi}\ast \widehat{Pf}).     
\]      
\end{lemma}
\begin{proof}
Let $\mu = f\cdot m$ be an absolutely continuous probability measure with 
density $f$.

By Definition \ref{def:transferoperator} we have that 
\[
L\mu = Pf \cdot m.     
\] 

By Lemma \ref{lemma:measuretoac}, we have that 
\[
\hat{\rho_{\xi}}\ast \widehat{L\mu} = (\hat{\rho_{\xi}}\ast \widehat{Pf})\cdot m,
\]
where the Lebesgue measure on the right handside is defined on $\mathbb{R}$.
Remark that by \ref{lemma:measuretoac} the support of $\hat{\rho}\ast \widehat{Pf}$
is contained in $[-1-\xi, 1+\xi]$.

Referring to Remark \ref{rem:abuse}, we have that 
\[
\pi_{*}(\hat{\rho_{\xi}}\ast \widehat{L\mu})=\pi_*(\hat{\rho_{\xi}}\ast \widehat{Pf})\cdot m
\]
where on the right handside $m$ is defined on $[-1,1]$.

Remark that by Lemmas \ref{lem:reg} and \ref{lem:proj},
and the fact that $P$ sends densities in densities, we have that $P_{\xi}$
is a well defined operator on densities.
\end{proof}

\begin{remark}
It is worth remarking that
$L_{\xi}\delta_y = \pi_*(\hat{\rho}_{\xi}(x-T(y)) \cdot m(x))$,
which, by Lemma \ref{lemma:propconvmeas}, Item \ref{it:delta} is 
absolutely continuous with respect to Lebesgue, with bounded variation density.
\end{remark}

\begin{definition}
Let $\mu_{\xi}$ be a fixed point for $L_{\xi}$, i.e.,
\[
L_{\xi}\mu_{\xi} = \mu_{\xi}.     
\]
We will call $\mu_{\xi}$ a \textbf{stationary measure} for $\mu_{\xi}$.
\end{definition}
\begin{remark}
If $P_{\xi}$ is the Perron-Frobenius operator operating on densities, and $f_{\xi}$
is a fixed point of this operator
\[
P_{\xi}f_{\xi} = f_{\xi} 
\]
then $\mu_{\xi}=f_{\xi}\cdot m$, where $m$ is the Lebesgue measure is a stationary measure. 
\end{remark}
The following theorem is a consequence of Birkhoff ergodic theorem and the skew product 
view of random dynamical systems, we refer to \cite{ViLLE}, and allows us to connect 
the notion of stationary measure and the notion of random dynamical system.
\begin{theorem}[Birkhoff Ergodic Theorem]
Suppose $L_{\xi}$ has a unique stationary measure $\mu_{\xi}$, let $\phi\in L^1(\mu_{\xi})$.
Then, for $\mu_{\xi}$ almost every initial condition $x_0$ and with probability one 
\[
\lim_{n\to +\infty}\frac{1}{n}\sum_{i=0}^{n-1}\phi(X_i)=\int \phi d\mu_{\xi}.     
\] 
\end{theorem}
\begin{remark}
We state the ergodic theorem in this weaker form, requiring uniqueness of the stationary measure
to simplify the treatment and avoid to define the notion of ergodicity for stationary measures.
\end{remark}
\begin{proof}[Sketch of proof]
It is possible to associate to our random dynamical system with additive noise 
a skew product $F:\Omega\times [-1,1]\to \Omega\times [-1,1]$, where 
$\Omega = [-\xi, \xi]^{\mathbb{N}}$, $\sigma: \Omega\to \Omega$ is the 
shift map and, for $\omega\in \Omega$, $x\in [-1,1]$ the skew product is 
defined as  
\[
F(\omega, x) = (\sigma \omega, \pi(T(x)+y)),     
\]
where $y=(\omega)_0$ is the first entry of $\omega$, and $\pi$ is the boundary condition. 

Denote by $\nu$ the product measure induced by $\rho_{\xi}\cdot m$ on $\Omega$, 
following the proof of
\cite[Proposition 5.4]{ViLLE} verbatim, we can see that $\mu_{\xi}$ is stationary 
if and only if $\nu\times \mu_{\xi}$ is invariant for $F$; 

We show the ``if'' claim; let $\psi(\omega, x)$ be a measurable function 
on $\Omega\times [-1,1]$, and let $\phi(x) = \int \psi(\omega, x)d\nu(\omega)$;
then 
\begin{align}
\nonumber \int\int \psi(\omega, x)&d\nu(\omega)d\mu_{\xi}(x) = \int \phi(x) d\mu_{\xi}(x)\\ 
\nonumber &=\int \phi(x) dL_{\xi} \mu_{\xi}(x)=\int\int\phi(\pi(T(x)+y))\rho_{\xi}(y)dm(y) d\mu_{\xi}(x)\\
&=\int\int\int \psi(\omega, \pi(T(x)+y))d\nu(\omega) \rho_{\xi}(y)dm(y) d\mu_{\xi}(x),\label{eq:skew}
\end{align}
since the product measure $\nu$ is invariant for the shift and by definition 
\[
\nu = (\rho_{\xi}\cdot m) \otimes \nu,
\] 
we have that \eqref{eq:skew} is equal to
\begin{align*}
\int\int \psi(\sigma(\omega), \pi(T(x)+y)) d\nu(\omega) d\mu_{\xi}(x) = \int\int \psi dL_{F}(\nu\times \mu)
\end{align*}
We show the ``only if'' claim; let $\phi:[-1,1]\to \mathbb{R}$ be bounded and measurable,
define $\psi(\omega, x) = \phi(x)$, then, recalling Remark \ref{rem:extbounded}
we have
\begin{align*}
\int \phi(x)d(L_{\xi}\mu_{\xi})(x) &= \int\int \phi(\pi(T(x)+y))\hat{\rho}_{\xi}(y)dm(y)d\mu_{\xi}(x)\\
&=\int\int \psi(\omega, \pi(T(x)+y))\hat{\rho}_{\xi}(y)dm(y)d\mu_{\xi}(x)\\
&=\int\int \psi(\sigma\omega, \pi(T(x)+y))\hat{\rho}_{\xi}(y)dm(y)d\mu_{\xi}(x)\\
&=\int \psi(\omega, x) dL_F(\nu\times \mu_{\xi})\\
&=\int\int  \psi(\omega, x) d(\nu \times\mu_{\xi}) = \int \phi(x)d\mu_{\xi}.
\end{align*} 
By \cite[Theorem 5.13]{ViLLE} and unicity of $\mu_{\xi}$ we get that 
$\mu_{\xi}$ is an ergodic stationary measure (we refer to \cite{ViLLE}
Section 5.3 for a definition),
therefore $\nu\times \mu_{\xi}$ is an ergodic invariant measure for $F$ 
and the statement follows.
\end{proof}

\subsection{Regularization properties of the annealed transfer operator}
The ergodic theorem tells us that if we want to understand the statistical properties
of a random dynamical system, we need to prove uniqueness of its stationary measure and
study its properties.
Our plan is to show that under some assumptions the random dynamical system admits 
a unique stationary measure, with density of bounded variation.

\begin{corollary}\label{cor:regularization}
The operator $P_{\xi}$ is a bounded operator from $L^1$ to $BV$, such that
\[
\Var_{[-1,1]}(P_{\xi}f)\leq \left(\Var_{[-\xi, \xi]}(\rho_{\xi})+2 \sup_{[-\xi, \xi]}|\rho_{\xi}(x)|\right)||f||_{L^1([-1,1])},     
\]
which in turn implies that $\Var_{[-1,1]}(P_{\xi}f)\leq 3||\rho_{\xi}||_{BV([-\xi,\xi])}||f||_{L^1([-1,1])}$.
\end{corollary}
\begin{proof}
This follows from Lemma \ref{lem:reg} and the proof of Lemma \ref{lem:proj},
i.e.,
\begin{align*}
\Var_{[-1, 1]}(P_{\xi}f) &= \Var_{[-1, 1]}(\pi_* (\hat{\rho}_{\xi}\ast \widehat{Pf}))\\
&\leq \Var_{[-1-\xi, -1+\xi]}(\hat{\rho}_{\xi}\ast \widehat{Pf})\\
&\leq \left(\Var_{[-\xi, \xi]}(\rho_{\xi})+2 \sup_{[-\xi, \xi]}|\rho_{\xi}(x)|\right)||Pf||_{L^1([-1,1])} 
\end{align*}
and the fact that $||P||_{L^1 \to L^1}\leq 1$.
As in many other occasions, we use that $||\rho_{\xi}||_{BV([-\xi, \xi])}\geq \sup_{[-\xi, \xi]}|\rho_{\xi}(x)|$ to give the following bound
\[
\left(\Var_{[-\xi, \xi]}(\rho_{\xi})+2 \sup_{[-\xi, \xi]}|\rho_{\xi}(x)|\right)\leq 3||\rho_{\xi}||_{BV([-\xi, \xi])}.
\]
\end{proof}
\begin{remark}
In particular, if $f$ is a density (Definition \ref{def:density}), we have that 
\[
\Var_{[-1, 1]}(P_{\xi}f)\leq \Var_{[-\xi, \xi]}(\rho_{\xi})+2 \sup_{[-\xi, \xi]}|\rho_{\xi}(x)|.
\]
\end{remark}

\begin{corollary}[Big noise amplitude limit] \label{cor:bignoise}
Let $f_{\xi}$ be a density which is a fixed point of $P_{\xi}$; then 
\begin{align*}
\Var_{[-1,1]}(f_{\xi})&\leq \left(\Var_{[-\xi, \xi]}(\rho_{\xi})+2 \sup_{[-\xi, \xi]}|\rho_{\xi}(x)|\right)||f_{\xi}||_{L^1([-1,1])}\\
&= \left(\Var_{[-\xi, \xi]}(\rho_{\xi})+2 \sup_{[-\xi, \xi]}|\rho_{\xi}(x)|\right).
\end{align*}
Moreover this implies that 
\[
\lim_{\xi\to+\infty} \Var_{[-1,1]}(f_{\xi}) = 0,
\]
and therefore 
\[
\lim_{\xi\to +\infty}||f_{\xi}-\frac{1}{2}||_{BV([-1,1])}=0.
\]
\end{corollary}
\begin{proof}
This follows from Corollary \ref{cor:regularization}:
\[
\Var_{[-1,1]}(f_{\xi}) = \Var_{[-1,1]}(P_{\xi}f_{\xi})\leq \left(\Var_{[-\xi, \xi]}(\rho_{\xi})+2 \sup_{[-\xi, \xi]}|\rho_{\xi}(x)|\right).
\]
The second statement follows from
\begin{align*}
\Var_{[-1,1]}(f_{\xi})&\leq \left(\Var_{[-\xi, \xi]}(\rho_{\xi})+2 \sup_{[-\xi, \xi]}|\rho_{\xi}(x)|\right)\\
&= \frac{1}{\xi}\left(\Var_{[-1,1]}(\rho)+2 \sup_{[-1, 1]}|\rho(x)|\right),
\end{align*}
so 
\[
\lim_{\xi \to+\infty} \Var_{[-1,1]}(f_{\xi}) = 0,
\]
which implies the thesis.
\end{proof}
\begin{remark}
Remark that Corollary \ref{cor:regularization} and \ref{cor:bignoise} do not depend on our choice 
of boundary condition. 
\end{remark}
\begin{remark}
Corollary \ref{cor:bignoise} tells us that for any bounded variation noise kernel,
as the amplitude of the noise increases, the orbits of the random dynamical system 
distribute themselves uniformly in the interval $[-1,1]$.
\end{remark}
\begin{definition}\label{def:averagezero}
Let 
\[
\mathcal{U}_0 = \{f \in L^1([-1, 1]) \mid \int f dm = 0\}.
\]
We call $\mathcal{U}_0$ the vector subspace of average $0$ measures; by abuse of notation 
we denote by $\mathcal{U}_0$ also its intersection with $BV([-1,1])$. 
We say $P_{\xi}$ \textbf{contracts the space of average zero functions in $L^1$} 
if 
\[
||P^n_{\xi}|_{\mathcal{U}_0}||_{L^1\to L^1}\leq C\theta^n
\]
for constants $C>0$, $0<\theta<1$.
We say $P_{\xi}$ \textbf{contracts the space of average zero functions in $BV$} 
if 
\[
||P^n_{\xi}|_{\mathcal{U}_0}||_{BV\to BV}\leq \tilde{C}\tilde{\theta}^n
\]
for constants $\tilde{C}>0$, $0<\tilde{\theta}<1$.
\end{definition}
\begin{lemma}
The operator $P_{\xi}$ contracts the space of average zero functions in $L^1$ if and only if 
it contracts the space of average of average zero functions in $BV$.
\end{lemma}
\begin{proof}
By Lemma \ref{lem:reg}, we have that
\[
||P_{\xi}||_{L^1\to BV}\leq 3||\rho_{\xi}||_{BV([-\xi,\xi])}.     
\]
Suppose $P_{\xi}$ contracts the space of average $0$ functions in $BV$.
Let $f$ be an average $0$ function in $L^1$, then
\[
||P^n_{\xi}f||_{L^1}\leq ||P^n_{\xi}f||_{BV} \leq ||P^{n-1}_{\xi}|_{\mathcal{U}_0}||_{BV \to BV} 3||\rho_{\xi}||_{BV([-\xi,\xi])} ||f||_{L^1}, 
\]
which implies that 
\[
||P^n_{\xi}|_{\mathcal{U}_0}||_{L^1}\leq \frac{3\tilde{C}||\rho_{\xi}||_{BV([-\xi,\xi])}}{\tilde{\theta}}\tilde{\theta}^{n},
\]
i.e., $P_{\xi}$ contracts the space of average $0$ functions in $L^1$.
If $P_{\xi}$ contracts the space of average $0$ functions in $L^1$ we have that, if $f$
is an average $0$ function in $BV$
\[
||P^n_{\xi} f||_{BV}\leq 3||\rho_{\xi}||_{BV([-\xi,\xi])} ||P^{n-1}_{\xi}f||_{L^1\to L^1}\leq 3||\rho_{\xi}||_{BV([-\xi,\xi])}\cdot C\theta^{n-1}||f||_{L^1};
\]
since $||f||_{L^1}\leq ||f||_{BV}$ this implies that
\[
||P^n_{\xi}|_{\mathcal{U}_0}||_{BV}\leq \frac{3||\rho_{\xi}||_{BV([-\xi,\xi])}\cdot C}{\theta}\theta^{n},
\]
i.e., that $P_{\xi}$ contracts the space of average functions in $BV$.
\end{proof}

\begin{lemma}\label{lemma:uniqueness}
If $P_{\xi}$ contracts the space of average $0$ functions in $L^1$ (or equivalently in $BV$), then
$L_{\xi}$ has a unique stationary measure.
\end{lemma}
\begin{proof}
We prove by contradiction that the stationary measure is unique: let $\mu$ and $\nu$
be stationary measures.
Since $L_{\xi}\mu=\mu$ and $L_{\xi}\nu =\nu$ we have that $\mu$ and $\nu$ are absolutely continuous
with respect to Lebesgue, with densities $f$ and $g$ respectively.
Now, $P_{\xi}f = f$ and $P_{\xi}g=g$, and, since $P_{\xi}$ contracts the space of 
average $0$ measures, we have that for any $n$
\[
||f-g||_{L^1} = ||P^n_{\xi}(f-g)||_{L^1}\leq C \theta^n ||f-g||_{L^1}.
\]
Take $N$ such that $C \theta^N<1$, the inequality above then implies that $||f-g||_{L^1}=0$,
which in turn implies that $\mu=\nu$.
\end{proof}

We will now generalize of a result in \cite{GaMoNi}: 
if for some noise amplitude the operator contracts the space of average $0$ functions in $L^1$, 
then for all bigger amplitudes the annealed operators also contracts the space of average $0$ functions in $L^1$.
We start by an auxiliary Lemma and Corollary.

\begin{lemma}\label{lemma:equivalent}
Let $\rho$ be a mother noise kernel, $\rho_{\xi}$ its rescaling, $\mu$ be a probability measure on $[-1, 1]$,
$\hat{\mu}$ its extension to $\mathbb{R}$ by $\hat{\mu}(A)=\mu(A\cap [-1,1])$,
then, for any measurable subset $A$ and for each $\hat{\xi}>\xi$ we have that
\[
\pi_*(\hat{\rho}_{\hat{\xi}}\ast \hat{\mu})(A)\geq \frac{c}{||\rho||_{BV}}\frac{\xi}{\hat{\xi}}\pi_*(\hat{\rho}_{\xi}\ast \hat{\mu})(A).      
\]
\end{lemma}
\begin{proof}
We remember that $\rho_{\xi}(x) = \frac{1}{\xi} \rho(x/\xi)$ and that $\rho(x)\geq c>0$
for all $x\in [-1, 1]$ by Definition of \ref{def:mothernoise}.
Therefore 
\[
\rho_{\xi}(x)\geq \frac{c}{\xi}
\]
for all $x \in [-\xi, \xi]$.

We have that 
\[
\pi_*(\hat{\rho}_{\hat{\xi}}\ast \hat{\mu})(A) = \int_{[-\hat{\xi}, \hat{\xi}]}\hat{\rho}_{\hat{\xi}}(y)\hat{\mu}(A-y) dm(y); 
\]
by the observation above we have that
\[
\int_{[-\hat{\xi}, \hat{\xi}]}\hat{\rho}_{\hat{\xi}}(y)\hat{\mu}(A-y) dm(y)\geq \frac{c}{\hat{\xi}} \int_{[-\hat{\xi}, \hat{\xi}]}\hat{\mu}(A-y) dm(y).
\]

Now, since 
\[
||\rho_{\xi}||_{\infty}\leq ||\rho_{\xi}||_{BV}\leq \frac{||\rho||_{BV}}{\xi}
\]
we have that 
\[
\pi_*(\hat{\rho}_{\xi}\ast \hat{\mu})(A)\leq \frac{||\rho||_{BV}}{\xi}\int_{[-\xi, \xi]}\hat{\mu}(A-y) dm(y).
\]

Since $\hat{\xi}>\xi$ and $\hat{\mu}(A-y)$ is nonnegative for all $y$ we have
\begin{align*}
\int_{[-\hat{\xi}, \hat{\xi}]}\hat{\rho}_{\hat{\xi}}(y)\hat{\mu}(A-y) dm(y)&\geq \frac{c}{\hat{\xi}} \int_{[-\hat{\xi}, \hat{\xi}]}\hat{\mu}(A-y) dm(y)\geq\\
\frac{c}{\hat{\xi}} \int_{[-\xi, \xi]}\hat{\mu}(A-y) dm(y)&\geq \frac{c}{||\rho||_{BV}}\frac{\xi}{\hat{\xi}}\int_{[-\xi, \xi]}\hat{\rho}_{\xi}(y)\hat{\mu}(A-y) dm(y).
\end{align*}
and the thesis follows.
\end{proof}

\begin{corollary}\label{cor:dens_lower_bound}
Let $\nu$ be a probability measure in $[-1,1]$.
Then, letting $\tau = \frac{c}{||\rho||_{BV}}\frac{\xi}{\hat{\xi}}$, we have that
\[
L_{\hat{\xi}}\mu(A)\geq \tau L_{\xi}\mu(A),     
\]
for all measurable subset $A$.
\end{corollary}
\begin{proof}
Use Lemma \ref{lemma:equivalent} with $\mu = L\nu$.
\end{proof}

We can now prove that mixing for some noise amplitude implies mixing for all bigger noise amplitudes.

\begin{lemma}\label{lemma:miximpliesmix}
Suppose $P_{\xi}$ contracts the space of average $0$ functions in $L^1$ for $\xi>0$;
then $P_{\hat{\xi}}$ contracts the space of average $0$ functions for any $\hat{\xi}>\xi$.
\end{lemma}
\begin{proof}
% Let 
% \[
% X_{\xi} = T(x)+\Omega_{\xi} \quad X_{\hat{\xi}} = T(x)+\Omega_{\hat{\xi}}.
% \]

% Taking $\nu=\chi_{[a,b]}dm$ in Lemma \ref{lemma:equivalent} we have there exists a constant $\tau$ (depending on $\hat{\xi}$)
% such that for all $a,b$ where $-\xi\leq a<b\leq \xi$ we have that
% \[
% \int_a^b \rho_{\hat{\xi}}(x) dm \geq \tau \int_a^b \rho_{\xi}(x)dm.
% \]

% Therefore we have that for any measurable subset $A$, and any $x\in [-1,1]$
% \[
% (L_{\hat{\xi}}\delta_x)(A) = \int_A \rho_{\hat{\xi}}(y-T(x)) dy
% \geq \tau \int_A \rho_{\xi}(y-T(x)) dy = \tau(L_{\xi}\delta_x)(A).      
% \]

By Lemma \ref{lemma:propconvmeas}, Item \ref{it:delta}, we know that $L_{\xi}\delta_x$ is 
an absolutely continuous probability measure and by Corollary \ref{cor:dens_lower_bound} we have that 
for any measurable subset $A$, and any $x\in [-1,1]$: 
\[
(L^m _{\hat{\xi}}\delta_x)(A)\geq \tau^m (L^m_{\xi}\delta_x)(A). 
\]

By hypothesis, $P_{\xi}$ is contracting the space of average $0$ functions in $L^1$,
so that, for any $x\in [-1,1]$, if $f_{\xi}$ is the density of the stationary measure and 
$g$ is the density of $L^m_{\xi}\delta_x$, we have that
\[
||f_{\xi}-P_{\xi}^{m-1}g||_{L^1} = ||P_{\xi}^{m-1}(f_{\xi}-g)||_{L^1} \leq 2C \theta^{m-1},     
\]
which in turn implies that for any measurable subset $A$
\[
|(L^m_{\xi}\delta_x)(A)-\mu_{\xi}(A)|\leq 2C \theta^{m-1}.
\]

Let $N$ such that $2C \theta^{N-1}<1$, and let
\[
\nu_N = \tau^N (1-2C \theta^{N-1})\mu_{\xi},
\]
then
\[
(L^N _{\hat{\xi}}\delta_x)(A)\geq \nu_N(A)
\]
for all $x$ and for all measurable $A$.

Remember that, if $\mu$ and $\nu$ are absolutely continuous measures 
with respect to Lebesgue, with densities $f$ and $g$ respectively, we have that 
the total variation norm for measures (we refer to \cite{MeyTwee} for its definition) is related by 
the $L^1$ norm  by the following equation:
\[
||\mu-\nu||_{TV} = \frac{1}{2}||f-g||_{L^1}.     
\]
Then, by \cite[Theorem 16.2.4]{MeyTwee} and the fact that $L_{\xi}$ maps measure into 
absolutely continuous measures, i.e., item 2 in Lemma \ref{lemma:measuretoac}, we have that 
\[
||P^n_{\hat{\xi}}|_{\mathcal{U}_0}||_{L^1}\leq 4\rho^{\lfloor n/N \rfloor}; 
\]
where $\rho = 1-\tau^N (1-2C \theta^{N-1})$.
\end{proof}

\begin{lemma}\label{lemma:biggerthan1}
There exists $C>0$, $0<\theta<1$ such that for all $\xi\geq 1$
\[
     ||P^n_{\xi}|_{\mathcal{U}_0}||_{L^1}\leq C\theta^n.
\]
\end{lemma}
\begin{proof}
Recall that $\rho_{\xi}(x) =\frac{1}{\xi}\rho(x/\xi)$. Therefore,
for $\xi$ in $[k, k+1)$, where $k$ is a positive natural number, we have that 
\[
\rho_{\xi}(x)\geq \frac{c}{k+1},      
\]
and we have that for all $x\in [-1,1]$, due to boundary conditions, 
\[
L_{\xi}\delta_x (A) \geq \frac{k}{k+1} c \cdot m(A).     
\]

This implies that for all $\xi\geq 1$, we have that 
\[
L_{\xi}\delta_x (A) \geq \frac{c}{2} \cdot m(A).
\]

By \cite[Theorem 16.2.4]{MeyTwee} and the fact that $L_{\xi}$ maps measure into 
absolutely continuous measures, i.e., Lemma \ref{lemma:measuretoac}, we have that 
\[
||P^n_{\xi}|_{\mathcal{U}_0}||_{L^1}\leq 4\rho^n; 
\]
where $\rho = 1-c/2$.
\end{proof}
\begin{corollary}\label{cor:unifcontr}
Suppose $P_{\xi}$ contracts the space of average $0$ functions in $L^1$ for $\xi>0$;
then there exists $C>0, 0<\theta<1$ such that
\[
||P^n_{\hat{\xi}}|_{\mathcal{U}_0}||_{L^1} \leq C\theta^n     
\]
for all $\hat{\xi}>\xi$.
\end{corollary}
\begin{proof}
Using a compactness argument, for all $\hat{\xi}$ in $(\xi,1]$ we have a uniform bound from Lemma \ref{lemma:miximpliesmix};
for $\xi>1$ we have a uniform bound from Lemma \ref{lemma:biggerthan1}.
\end{proof}

\subsection{$L^r$ continuity of the stationary measure with respect to the noise size}
In this subsection we prove continuity in $L^r([-1,1])$ of 
the stationary measure with respect to the noise size at a 
fixed noise size $\xi>0$.

\begin{lemma}\label{lemma:continuity}
Suppose that $P_{\xi}$ contracts the space of average $0$ functions in $L^1$.
Moreover,suppose that there exists a $0<\epsilon<\xi$ such that 
for all $\hat{\xi}$ in $(\xi-\epsilon, \xi+\epsilon)$ 
the operator $P_{\hat{\xi}}$ has a unique fixed density $f_{\hat{\xi}}$.

Then
\[
\lim_{\hat{\xi}\to \xi}||f_{\hat{\xi}}-f_{\xi}||_{L^1}=0.
\]
\end{lemma}
\begin{proof}
A bounded variation function on the interval has a countable set of 
discontinuity points, since it can be written as the difference of two monotone
functions \cite{AmFuPa}.

Moreover, a bounded variation function is bounded; fix $\xi$ and let $\Omega_{\xi}$
be the set of discontinuities of $\rho_{\xi}$.
We claim that 
\[
\lim_{\hat{\xi}\to \xi}||\hat{\rho}_{\hat{\xi}}-\hat{\rho}_{\xi}||_{L^1} = 0,     
\]
where $\hat{\rho}_{\hat{\xi}}$, $\hat{\rho}_{\xi}$ are the extensions
of $\rho_{\hat{\xi}}$ and $\rho_{\xi}$ respectively to $\mathbb{R}$.  

Observe that 
\[
||\hat{\rho}_{\hat{\xi}}-\hat{\rho}_{\xi}||_{L^1}=\int_{[-\hat{\xi},\hat{\xi}]\setminus{\Omega_{\xi}}}|\hat{\rho}_{\hat{\xi}}(x)-\hat{\rho}_{\xi}|dm\leq \max(\frac{1}{\hat{\xi}}, \frac{1}{\xi})||\rho||_{BV}.
\]

By the dominated convergence theorem we have then that 
%\begin{align*}
% &\\
 %&+\int_{[-\xi, \hat{\xi})\lim_{\hat{\xi}\to \xi}|\hat{\rho}_{\hat{\xi}}(x)|dm + \int_{[-\hat{\xi},\hat{\xi}]\setminus{\Omega_{\xi}}} \lim_{\hat{\xi}\to \xi}|\hat{\rho}_{\hat{\xi}}(x)-\hat{\rho}_{\xi}|dm
%\end{align*}
\begin{align*}
\lim_{\hat{\xi}\to \xi} \int_{[-\hat{\xi},\hat{\xi}]\setminus{\Omega_{\xi}}}|\hat{\rho}_{\hat{\xi}}(x)-\hat{\rho}_{\xi}(x)|dm &= \int_{[-\hat{\xi},\xi) \cup (\xi, \hat{\xi}]}\lim_{\hat{\xi}\to \xi}|\hat{\rho}_{\hat{\xi}}(x)|dm+\\ 
&\int_{[-\xi,\xi]\setminus{\Omega_{\xi}}} \lim_{\hat{\xi}\to \xi}|\hat{\rho}_{\hat{\xi}}(x)-\hat{\rho}_{\xi}(x)|dm   
\end{align*}
which goes to $0$ as $\hat{\xi}$ goes to $\xi$.

By $L^1$ continuity of the convolution this in turn implies that
\[
\lim_{\hat{\xi}\to \xi}||P_{\hat{\xi}}-P_{\xi}||_{L^1\to L^1}=0,
\] 
and 
\begin{align*}
||f_{\xi}-f_{\hat{\xi}}||_{L^1}\leq ||P_{\xi}^N (f_{\xi}-f_{\hat{\xi}})||_{L^1}+||P_{\xi}^N f_{\hat{\xi}}-P_{\hat{\xi}}^N f_{\hat{\xi}}||_{L^1}.
\end{align*}

Since $||P^i_{\xi}|_{\mathcal{U}_0}||_{L^1}<C \theta^i$ where 
$0<\theta<1$, there exists a positive $N$ such that $C\theta^N<1/2$, since $f_{\xi}-f_{\hat{\xi}}$ is 
an average $0$ function in $L^1$ we have that
\begin{align*}
||f_{\xi}-f_{\hat{\xi}}||_{L^1}\leq \frac{1}{2}||f_{\xi}-f_{\hat{\xi}}||_{L^1}+ ||P_{\xi}^N f_{\hat{\xi}}-P_{\hat{\xi}}^N f_{\hat{\xi}}||_{L^1},
\end{align*}
and we estimate the right hand side by telescopizing the difference of powers:
\begin{align*}
||P_{\xi}^N f_{\hat{\xi}}-P_{\hat{\xi}}^N f_{\hat{\xi}}||_{L^1}\leq &\sum_{k=0}^{N-1} ||P^{k}_{\xi}|_{\mathcal{U}_0}||_{L^1\to L^1}||P_{\xi}-P_{\hat{\xi}}||_{L^1\to L^1} ||P^{N-k-1}_{\hat{\xi}} f_{\hat{\xi}}||_{L^1}\\
&\leq \sum_{k=0}^{N-1} ||P^{k}_{\xi}|_{\mathcal{U}_0}||_{L^1\to L^1}||P_{\xi}-P_{\hat{\xi}}||_{L^1\to L^1} ||f_{\hat{\xi}}||_{L^1}\\
&\leq \frac{C}{1-\theta}||P_{\xi}-P_{\hat{\xi}}||_{L^1\to L^1};
\end{align*}
where we used that $||f_{\hat{\xi}}||_{L^1}=1$.
This implies that   
\[
||f_{\xi}-f_{\hat{\xi}}||_{L^1}\leq 2C\frac{1}{1-\theta}||P_{\xi}-P_{\hat{\xi}}||_{L^1\to L^1}
\]
Taking the limit as $\hat{\xi}\to \xi$ we conclude the proof.
\end{proof}
\begin{remark}\label{rem:rightcont}
The same argument can be used to
prove right continuity of the stationary measure in $L^1$.
The main difference is that we can drop the hypothesis of the uniqueness of the stationary measure
since by Corollary \ref{cor:unifcontr} the contraction of the space of average $0$ functions at $\xi$ implies uniform contraction 
for all $\hat{\xi}\geq\xi$, so the uniqueness of $f_{\hat{\xi}}$ follows. 
\end{remark}
\begin{remark}
Lemma \ref{lemma:continuity} proves the continuity at $\xi$ in $L^1$ norm of the 
stationary density if the operator $P_{\xi}$ is contracting the space of average $0$ measures
and the stationary measure is unique in a neighborhood of $\xi$.

Even if we have uniform contraction rates, we can only prove continuity
in $L^1$ of the stationary density as a function of $\xi$.

More regular noise kernel allow us to prove stronger regularity of $f_{\xi}$
as a function of $\xi$, which reflects in stronger regularity of the Birkhoff 
averages of observables as a function of $\xi$.
\end{remark}

\begin{corollary}\label{cor:continuityLr}
Suppose that $P_{\xi}$ contracts the space of average $0$ functions in $L^1$.
Moreover,suppose that there exists a $0<\epsilon<\xi$ such that 
for all $\hat{\xi}$ in $(\xi-\epsilon, \xi+\epsilon)$ 
the operator $P_{\hat{\xi}}$ has a unique fixed density $f_{\hat{\xi}}$.
     
Then, for any $1<r<+\infty$
\[
\lim_{\hat{\xi}\to \xi}||f_{\hat{\xi}}-f_{\xi}||_{L^r}=0.
\]
\end{corollary}
\begin{proof}
Recall that if $f_{\xi}$ is a fixed point of $P_{\xi}$ 
with $||f_{\xi}||_{L^1}=1$ we have that 
\[
||f_{\xi}||_{BV} = ||P_{\xi}f_{\xi}||_{BV}\leq ||P_{\xi}||_{L^1\to BV} ||f_{\xi}||_{L^1}\leq 3||\rho_{\xi}||_{BV}.
\]

The $BV$ norm bounds from above the $L^{\infty}$-norm, so 
\[
||f_{\xi}-f_{\hat{\xi}}||_{\infty}\leq 6\frac{||\rho||_{BV}}{\xi-\epsilon}.    
\]

Therefore, $f_{\xi}-f_{\hat{\xi}}$ belongs to $L^1 \cap L^{\infty}$ and by the classical 
$L^p$ interpolation inequality we have that 
\[
||f_{\xi}-f_{\hat{\xi}}||_{L^r}\leq \left(||f_{\xi}-f_{\hat{\xi}}||_{L^1}\right)^{1/r}\cdot \left(||f_{\xi}-f_{\hat{\xi}}||_{L^{\infty}}\right)^{1-1/r}. 
\]

Therefore, 
\[
\lim_{\hat{\xi}\to \xi}||f_{\xi}-f_{\hat{\xi}}||_{L^r}\leq \left(6\frac{||\rho||_{BV}}{\xi-\epsilon}\right)^{1-1/r} \lim_{\hat{\xi}\to \xi} \left(||f_{\xi}-f_{\hat{\xi}}||_{L^1}\right)^{1/r} = 0.
\]

\end{proof}

\begin{corollary}\label{cor:continuityfunc}
Let $\phi\in L^p([-1,1])$, with $p>1$; suppose there exists a $\xi_0$ such that 
$P_{\xi_0}$ contracts the space of average $0$ functions in $BV$ or equivalently $L^1$.
Then, the function 
\[
A_{\phi}(\xi) = \int_{-1}^1 \phi d\mu_{\xi}      
\] 
is well defined and continuous for all $\xi\geq \xi_0$.
\end{corollary}
\begin{proof}
By Lemma \ref{lemma:miximpliesmix} we have that $P_{\xi}$ contracts the space
of average $0$ functions in $BV$ for all $\xi\geq \xi_0$
This implies by Lemma \ref{lemma:uniqueness} that for each $\xi\geq \xi_0$
there exists a unique stationary measure $\mu_{\xi}$ with density $f_{\xi}$ in $BV$.
By Lemma \ref{cor:continuityLr}, fixing $\epsilon< \xi-\xi_0$ and letting 
$q>0$ be such that $1/q+1/p=1$
we have that, for $\hat{\xi}\in (\xi-\epsilon, \xi+\epsilon)$
\[
\lim_{\hat{\xi}\to \xi}|\int_{-1}^1 \phi f_{\xi} dm-\int_{-1}^1 \phi f_{\hat{\xi}}dm| \leq \lim_{\hat{\xi}\to \xi}||\phi||_{L^p}||f_{\xi}-f_{\hat{\xi}}||_{L^q}=0     
\]
which implies the continuity of $A_{\phi}$ for all $\xi> \xi_0$.
The proof of \ref{lemma:continuity} can be redone verbatim for right continuity 
at $\xi_0$ as explained in remark \ref{rem:rightcont}, which implies right continuity 
at $\xi_0$ and the thesis.
\end{proof}

\begin{corollary}\label{cor:bignoiselimitfunc}
Let $\phi\in L^1([-1,1])$; suppose there exists an $0<\xi_0<+\infty$ such that 
$P_{\xi_0}$ contracts the space of average $0$ functions in $BV$ or equivalently $L^1$.
Then, if  
\[
A_{\phi}(\xi) = \int_{-1}^1 \phi d\mu_{\xi}      
\] 
we have that
\[
\lim_{\xi \to +\infty}A_{\phi}(\xi) = \int_{-1}^1 \phi \frac{1}{2}dm.     
\]
\end{corollary}
\begin{proof}
As in Corollary \ref{cor:continuityfunc} the function is well defined for all $\xi\geq \xi_0$.
We have that 
\[
\lim_{\xi\to +\infty} |\int_{-1}^1 \phi f_{\xi}dm -\int_{-1}^1 \phi \frac{1}{2}dm|\leq \lim_{\xi\to +\infty} ||\phi||_{L^1}||f_{\xi}-\frac{1}{2}||_{L^{\infty}}=0,
\]
recalling that the $BV$ norm bounds from above the $L^{\infty}$ norm, the thesis follows from Corollary \ref{cor:bignoise} .
\end{proof}

\subsection{Continuity with respect to the base dynamic $T$}
To study the behavior as the base dynamic varies, we will use the following arguments by M. Monge, 
that was proved for a version of \cite{GaMoNi}.

\begin{definition}
A piecewise continuous map $T$ on $[-1,1]$ is a function $T:[-1,1]\rightarrow
\lbrack -1,1]$ such that there is partition $\{I_{i}\}_{1\leq i\leq k}$ of $%
[-1,1]$ made of intervals $I_{i}$ such that $T$ has a continuous extension to
the closure $\bar{I}_{i}$ of each interval.
We call this partition the \textbf{continuity partition} of $T$.

If two piecewise continuous maps $T_1$ and $T_2$ share the same continuity partition
we define 
\[
||T_1-T_2||_{\infty} = \max_i \sup_{x\in I_i} |T_1(x)-T_2(x)|.     
\]
\end{definition}
\begin{remark}
Remark that a piecewise continuous map is uniformly continuous when restricted to
each $I_i$ in its continuity partition.
\end{remark}
\begin{remark}
The condition that two maps share the same continuity partition is used to 
generalize the sup distance on continuous maps to piecewise continuous maps; 
as observed by one of the referees the arguments in the rest of the section
do not depend strictly on it but the treatment is easier if 
we assume it.
\end{remark}

\begin{definition}
The Wasserstein-Kantorovich distance of two probability measures is defined as
\[
W(\mu, \nu)=\sup_{\Lip(\phi)\leq 1, ||\phi||_{\infty}=1}\left|\int \phi d\mu - \int \phi d\nu \right|
\]
\end{definition}
\begin{remark}
We refer to \cite{GaMoNi} for the properties of the Wasserstein-Kantorovich distance we use.
It is worth observing that
\[
W(\delta_p, \delta_q)=|p-q|.
\]
\end{remark}

We now give a proof of \cite[Lemma 51]{GaMoNi}, starting by proving 
another property of bounded variation functions.

\begin{lemma}
Let $\phi$ be a bounded variation function on $[a,b]$, zero outside of $[a,b]$. 
Let $\tau_h$ be the translation operator $\tau_h(\phi)(x)=\phi(x+h)$.
Then 
\[
||\tau_h \phi-\phi||_{L^1([a-h, b+h])}\leq h \left(\Var_{[a,b]}(\phi)+4\sup_{[a,b]}|\phi(x)|\right).
\] 
\end{lemma}
\begin{proof}
Without loss of generality, suppose $h>0$, the negative case is analogous.

We start by observing that
\begin{align*}
&||\tau_h(\phi)-\phi||_{L^1([a-h, b+h])}\\
&=\int_{a-h}^{a} |\phi(x+h)|dx + \int_{a}^{b-h}|\phi(x+h)-\phi(x)|dx+\int_{b-h}^b |\phi(x)|dx\\
&\leq \int_{a}^{b-h}|\phi(x+h)-\phi(x)|dx+2\sup_{[a,b]}|\phi(x)|\cdot h.
\end{align*}
     
Let now $N$ be the biggest integer such that $N h \leq b-h-a$;
then, the intervals $I_i = [a+i h, a+(i+1) h]$ for $i = 0, \ldots, N-1$,
and $J= [a+Nh, b-h]$ are a partition of $[a, b-h]$; remark that $m(J)<h$.
Then
\begin{align*}
\int_{a}^{b-h}&|\phi(x+h)-\phi(x)|dx = \sum_i \int_{I_i} |\phi(x+h)-\phi(x)|dx +\int_J |\phi(x+h)-\phi(x)|dx\\
&\leq \sum_{i=0}^{N-1} \int_0^{h}|\phi(a+(i+1)h+z)-\phi(a+ih+z)|dz + 2\sup_{[a,b]}|\phi(x)| \cdot h,     
\end{align*}
where on each $U_i$ we used the change of coordinates $x = a+ih+z$. Now, we have that 
\begin{align*}
\int_0^{h} \sum_i |\phi(a+(i+1)h+x)-\phi(a+ih+x)| dx \leq \int_{0}^{h} \Var_{[a+x, b+x]}(\tau_{x}\phi) dx.      
\end{align*}
since the variation is translation invariant, we have then that 
\[
\int_{a}^{b-h}|\phi(x+h)-\phi(x)|dx \leq \left(\Var_{[a,b]}(\phi)+ 2\sup_{[a,b]}|\phi(x)|\right)\cdot h
\]

Summarizing, using the fact that $\phi$ is zero outside of $[a,b]$, we have
\[
||\tau_h(\phi)-\phi||_{L^1([a-h, b+h])}\leq \left(\Var_{[a,b]}(\phi)+4\sup_{[a,b]}|\phi(x)|\right)h.
\]
\end{proof}

\begin{lemma}\label{lemma:convlip}
Let $\phi$ be a bounded variation function on $[a,b]$, zero outside of $[a,b]$, and let
$\psi\in L^{\infty}(\mathbb{R})$.
Then, their convolution $\phi\ast \psi$ is a Lipschitz function with Lipschitz constant 
bounded above by $\left(\Var_{[a,b]}(\phi)+4\sup_{[a,b]}|\phi(x)|\right)||\psi||_{\infty}$.
\end{lemma}
\begin{proof}
This follows from the definition of convolution
\begin{align*}
&|\phi\ast \psi(x+h)-\phi\ast \psi(x)| = \left|\int_{\mathbb{R}}\left(\phi(x+h-y)-\phi(x-y)\right)\psi(y)dy\right|\\
&\leq ||\psi||_{L^{\infty}}  \int_{\mathbb{R}} \left|\phi(x+h-y)-\phi(x-y)\right| dy = ||\psi||_{L^{\infty}}||\tau_h(\tau_x\phi)-\tau_x\phi||_{L^1(\mathbb{R})}\\
&= ||\psi||_{L^{\infty}}||\tau_h(\phi)-\phi||_{L^1([a-h, b+h])}\leq ||\psi||_{L^{\infty}}\left(\Var_{[a,b]}(\phi)+4\sup_{[a,b]}|\phi(x)|\right)h. 
\end{align*}
where we used the fact that $\phi$ is $0$ outside $[a,b]$ and invariance of the $L^1$ norm by $\tau_x$.
\end{proof}
\begin{lemma}
Let now $\mu$ and $\nu$ be probability measures on $[-1,1]$; as in Lemma \ref{lemma:measuretoac} we have that
$\hat{\rho}_{\xi}\ast \hat{\mu}$ and $\hat{\rho}_{\xi}\ast\hat{\nu}$ are absolutely continuous with respect to Lebesgue,
let $f$ and $g$ be their densities. Then 
\[
||f-g||_{L^1([-1-\xi, 1+\xi])}\leq \left(\Var(\rho_{\xi})+4\sup_{[-\xi,\xi]}|\rho_{\xi}(x)|\right)W(\mu, \nu)\leq 5||\rho_{\xi}||_{BV}W(\mu, \nu).
\]
\end{lemma}
\begin{proof}
Recall that $f$ and $g$ are $0$ outside $[-1-\xi, 1+\xi]$; for each $\psi\in L^{\infty}(\mathbb{R})$ we have that 
\begin{align*}
&\left|\int \psi (f-g) dx \right| = \left|\int \psi d(\hat{\rho}_{\xi}\ast \hat{\mu}) -\int \psi d(\hat{\rho}_{\xi}\ast \hat{\nu}) \right| \\
&= \left|\int \int  \psi(x) \hat{\rho}_{\xi}(x-y) dm(x)d\hat{\mu}(y) -\int \int  \psi(x) \hat{\rho}_{\xi}(x-y) dm(x)d\hat{\nu}(y) \right| \\
&= \left|\int (\psi\ast \hat{\rho}_{\xi})  d\hat{\mu} -\int (\psi\ast\hat{\rho}_{\xi})  d\hat{\nu} \right|.
\end{align*}

By Lemma \ref{lemma:convlip} and definition of Wasserstein distance we have then
\[
\left|\int \psi (f-g) dx \right|\leq \left(\Var_{[-\xi,\xi]}(\rho_{\xi})+4\sup_{[-\xi,\xi]}|\rho_{\xi}(x)|\right) ||\psi||_{L^{\infty}}\cdot W(\mu, \nu),
\]
which in turn implies the thesis by taking as $\psi$ the function with value $1$ if $f(x)\geq g(x)$ and 
value $-1$ if $f(x)<g(x)$.
\end{proof}
\begin{remark}
The constants in the preceding lemmas are not optimal, but are enough for our 
goal of studying the continuity of the stationary density with respect 
to the parameters of the system in the presence of positive amplitude 
noise.
\end{remark}

\begin{lemma}\label{lemma:wassT}
Let $T_{1}$ and $T_{2}:[-1,1]\rightarrow [-1,1]$ be piecewise
continuous nonsingular maps that share the same continuity partition 
and let $L_{T_{1}},L_{T_{2}}$ the associated transfer operators, 
let $f\in L^{1}$. Then:
\begin{equation*}
W(L_{T_{1}}(fdm), L_{T_{2}}(fdm))\leq ||T_{1}-T_{2}||_{\infty}||f||_{1},
\end{equation*}
or equivalently
\[
W((P_{T_{1}}f)dm, (P_{T_{2}}f)dm)\leq ||T_{1}-T_{2}||_{\infty}||f||_{1}.
\]
\end{lemma}
\begin{proof}
Let $[a,b]$ be an interval and let $\mathcal{P} = \{p_0=a, \ldots, p_n=b\}$
be the endpoints of a partition such that $p_{i+1}-p_i\leq D$ 
(in the following we will call $D$ the diameter of the partition).

Let $f\in L^1([a,b])$ be a positive function; 
the projection of $f dm$ associated to the partition $\mathcal{P}$ is  
\[
\pi_{\mathcal{P}}f = \sum_{i=0}^{n-1} \left(\int_{p_i}^{p_{i+1}} f dm\right) \cdot \delta_{p_i}
\]
where $\delta_{p_i}$ is the Dirac $\delta$ at $p_i$.

Then, for any Lipschitz function $\phi$ on $[a,b]$ we have
\[
\left|\sum_{i=0}^{n-1} \int_{p_i}^{p_{i+1}}\phi(p_i)f(x)dm - \int_{p_i}^{p_{i+1}}\phi(x)f(x)dm\right|\leq \textrm{Lip}(\phi)||f||_{L^1([a,b])}\cdot D,     
\]
which implies
\[
W(\pi_{\mathcal{P}}f, fdm)\leq D ||f||_{L^1([a,b])}.
\]

Fix $\epsilon>0$, by uniform continuity there exists a $D$ such that the image of a partition 
of diameter $D$ has diameter at most $\epsilon$; 
let $f$ be a density, and let $\mathcal{P}$ be a partition of diameter $D$, as above.

For a Dirac $\delta_p$ at $p$, we have that 
\[
L_{T_i}\delta_p = \delta_{T_i(p)}
\]
for $i=1,2$, which implies
\[
W(L_{T_1}\delta_p,L_{T_2}\delta_p)\leq ||T_1-T_2||_{\infty}
\]

By triangle inequality, this implies that
\[
W(L_{T_1}\pi_{\mathcal{P}} f, L_{T_2}\pi_{\mathcal{P}} f)\leq \sum_i  \left(\int_{p_i}^{p_{i+1}}fdm\right)  \cdot W(L_{T_1}\delta_{p_i}, L_{T_2}\delta_{p_i})\leq ||T_1-T_2||_{\infty}||f||_{L^1}.   
\]

Now, any Lipschitz function $\phi$ is bounded; by the duality properties of the transfer operator
and the Koopman operator, we have
\[
\left|\int \phi L_{T_1}(f dm)-\sum_i \phi(T_1(p_i))\int_{p_i}^{p_{i+1}} f dm\right|=\left|\sum_i \int_{p_i}^{p_{i+1}} (\phi\circ T_1(x)- \phi\circ T_1(p_i)) f dm \right|,     
\]
which in turn implies
\[
W(L_{T_1}(f dm), L_{T_1}\pi_{\mathcal{P}} f)\leq \epsilon,
\]
and similarly for $T_2$.

This implies that
\[
W(L_{T_1}(f dm), L_{T_2}(f dm))\leq W(L_{T_1}\pi_{\mathcal{P}} f, L_{T_2}\pi_{\mathcal{P}} f)+2\epsilon,     
\]
as $\epsilon$ is arbitrary, we obtain the thesis.
\end{proof}

\begin{definition}
Let $T_{1}$ and $T_{2}:[-1,1]\rightarrow \lbrack -1,1]$ be
piecewise continuous nonsingular maps that share the same  
continuity partition.
We will denote by
\[
L_{\xi, T_i}= \pi_*(\rho_{\xi }\ast L_{T_{i}}), \quad \textrm{for } i=1,2,
\]
and the associated annealed Perron-Frobenius operators
\begin{equation*}
P_{\xi ,T_{i}}f=\pi_*(\rho_{\xi }\ast P_{T_{i}}(f)) \quad \textrm{for } i=1,2.
\end{equation*}
     
\end{definition}

\begin{lemma}\label{prt}
Let $T_{1}$ and $T_{2}:[-1,1]\rightarrow \lbrack -1,1]$ be
piecewise continuous nonsingular maps that share the same  
continuity partition.
Then for any $f\in L^{1}$:
\begin{equation*}
||P_{\xi ,T_{1}}(f)-P_{\xi ,T_{1}}(f)||_{1}\leq ||T_{1}-T_{2}||_{\infty
}5||\rho _{\xi }||_{BV}||f||_{1}
\end{equation*}
\end{lemma}
\begin{proof}
\begin{eqnarray*}
||P_{\xi ,T_{1}}(f)-P_{\xi ,T_{2}}(f)||_{1} &=&||\pi_*||_{BV \to BV}||\hat{\rho} _{\xi }\ast%
(P_{T_{1}}(f)-P_{T_{2}}(f))||_{1} \\
&\leq &5||\rho _{\xi }||_{BV}\cdot W(P_{T_{1}}(f)dm,P_{T_{2}}(f)dm),
\end{eqnarray*}%
where the operator norm of $\pi_*:{BV([-1-\xi, 1+\xi])\to BV([-1,1])}$ is bounded by Lemma \ref{lem:proj}.
The statement then follows by Lemma \ref{lemma:wassT}.
\end{proof}

\begin{lemma}\label{lem:contmap}
Let $T_1$ and $T_2$ be
piecewise continuous nonsingular maps that share the same  
continuity partition.
Suppose $P_{\xi, T_1}$ contracts the space of average $0$ functions in $L^1$
with constants $C>0$ and $0<\theta<1$ then 
\[
||P_{\xi, T_1}^n f-P_{\xi, T_2}^n f||_{L^1}\leq \frac{C}{1-\theta}||T_1-T_2||_{\infty}5||\rho_{\xi}||_{BV}.
\]
\end{lemma}
\begin{proof}
This follows from a telescopization argument:
\[
||P_{\xi, T_1}^n f-P_{\xi, T_2}^n f||_{L^1}\leq \sum_{i=0}^n ||P_{\xi, T_1}^i||_{L^1\to L^1}||P_{\xi, T_1}-P_{\xi, T_2}||_{L^1\to L^1}||P^{n-i-1}_{\xi, T_2}f||_{L^1}.
\] 
Since $||P_{\xi, T_2}f||_{L^1}\leq ||f||_{L^1}$ and by Lemma \ref{prt}, we have that
\[
||P_{\xi, T_1}^n f-P_{\xi, T_2}^n f||_{L^1}\leq \sum_{i=0}^n C\theta^i ||T_{1}-T_{2}||_{\infty
     }5||\rho _{\xi }||_{BV}||f||_{1}
\]
and the thesis follows.
\end{proof}

\section{Proof of Theorem \ref{thm:result}}\label{sec:suff}
The results in Section \ref{sec:annealed} already allow us to prove Theorem \ref{thm:result}.

\begin{proof}[Proof of Theorem \ref{thm:result}]

Hypothesis R2 and R4 together with Lemma \ref{lemma:miximpliesmix}
prove that for all $\xi\geq \xi_1$ the operator $P_{\xi}$ contracts the space
of average $0$ functions in $BV$ (and equivalently in $L^1$).
This guarantees uniqueness of the stationary measure.

Hypothesis R3 guarantees that $\ln(|T'|)\in L^p(m)$ for $p> 1$; 
by Corollary \ref{cor:continuityfunc} together with 
Hypothesis R1 this allows us to prove that the function
\[
\lambda(\xi) = \int_{-1}^1 \ln(|T'|)d\mu_{\xi}     
\]
is well defined and continuous in $[0, +\infty)$.

Corollary \ref{cor:bignoiselimitfunc} together with Hypothesis D2 
and R3 allow us to state that
\[
\lambda(0)>0, \quad \lim_{\xi\to +\infty}\lambda(\xi)<0,
\]
therefore our system shows Noise Induced Order.
\end{proof}

\section{Consequences for the model}\label{sec:consequences}
In this section, the noise kernel is
\[
\rho(x)=\frac{1}{2}\chi_{[-1, 1]}
\]
the (normalized) characteristic function of the interval $[-1,1]$.

The family $T_{\alpha, \beta}:[-1,1]\to [-1, 1]$ is defined by
\begin{equation}\label{eq:model}
T_{\alpha, \beta}(x) = 1-2\beta|x|^{\alpha}.
\end{equation}

\subsection{Deterministic behavior}
The family $T_{\alpha, \beta}$ for $\alpha\geq 2$ and $0< \beta\leq 1$
is a family of unimodal maps, a classical example of non-uniformly hyperbolic 
dynamics.

In this family, the prototypical example is the quadratic family, i.e.,
$T_{2, \beta}$ as $\beta$ varies; the long term behavior of the system is strongly 
sensitive with respect to the parameter $\beta$: 
outside a parameter set of Lebesgue measure $0$ (the infinitely renormalizable parameters \cite{Lyu}), 
the parameters can be classified into two categories:
\begin{itemize}
\item a dense subset of \textbf{regular parameters} where all the points converge to a periodic attracting orbit
\item a positive measure Cantor set of \textbf{stochastic parameters} that admit an absolutely 
continuous invariant probability measure and have positive Lyapunov exponent.
\end{itemize} 

It is worth discussing the properties of the Schwarzian derivative 
for the family $T_{\alpha, \beta}$.
\begin{lemma}
For $\alpha>1$ the Schwarzian derivative of $T_{\alpha, \beta}$ is well defined and 
negative in $[-1,0) \cup (0, 1]$.
\end{lemma}
\begin{proof}
This follows from computation:
\[
S(T_{\alpha, \beta}) = \left(\frac{T''_{\alpha, \beta}}{T'_{\alpha, \beta}}\right)'-\frac{1}{2}\left(\frac{T''_{\alpha, \beta}}{T'_{\alpha, \beta}}\right)^2.     
\]
For $x>0$, we have that 
\[
\frac{T''_{\alpha, \beta}}{T'_{\alpha, \beta}}(x) = \frac{-2\beta\alpha(\alpha-1)x^{\alpha-2}}{-2\beta\alpha x^{\alpha-1}}=\frac{\alpha-1}{x},\quad \frac{T'''_{\alpha, \beta}}{T'_{\alpha, \beta}}(x)=\frac{(\alpha-1)(\alpha-2)}{x^2}.
\]
for $x<0$, similarly we have that
\[
\frac{T''_{\alpha, \beta}}{T'_{\alpha, \beta}}(x) = \frac{-2\beta\alpha(\alpha-1)(-x)^{\alpha-2}}{2\beta\alpha (-x)^{\alpha-1}}=-\frac{\alpha-1}{(-x)}=\frac{\alpha-1}{x}.
\]
and 
\[
\frac{T'''_{\alpha, \beta}}{T'_{\alpha, \beta}}(x) = \frac{2\beta\alpha(\alpha-1)(\alpha-2)(-x)^{\alpha-3}}{2\beta\alpha (-x)^{\alpha-1}}= \frac{(\alpha-1)(\alpha-2)}{x^2}.
\]

Therefore, for $x\in [-1,0)\cup(0, 1]$ we have that 
\[
S(T_{\alpha, \beta})(x)= \frac{(\alpha-1)(\alpha-2)}{x^2}-\frac{3}{2}\frac{(\alpha-1)^2}{x^2} = -\frac{1}{2}\frac{\alpha^2-1}{x^2}<0.
\]
\end{proof}
\begin{remark}
The unique point where the Schwarzian derivative of $T_{\alpha,\beta}$ is not 
defined is the critical point. 
This observation is not new, and is used extensively in \cite{Me}; intuitively,
due to the slow recurrence of the critical orbit to the critical point,
the levels of our tower are avoiding the critical point. 

Therefore, when we build the induction scheme, the Schwarzian derivative of the iterates 
is going to be definite and negative.
\end{remark}

We will prove now that our systems, when $\alpha\geq 2$ and $\beta=1$ satisfy the 
hypothesis of \cite[Theorem I.5]{TheTreYou}.
This permits us to state that $\beta$ is a density point of \textbf{stochastic} parameters,
i.e., parameters that admit an a.c.i.p. and have positive Lyapunov exponent.

To avoid notation clutter, in some of the following equations we are going to use the notation
$f_{\beta}(x):=f(\beta, x)$, and the notation $c_n(\beta) := f_{\beta}^n(0)$ for 
the critical orbit.

\begin{definition}
We say $f(\beta, x)$ is a \textbf{regular family} if 
\begin{enumerate}
\item $f(\beta, x)$ is $C^2$ in $x, \beta$;
\item $c=0$ is the unique critical point of $f(\beta, x)$, $f(\beta, x)$ is increasing on 
$[-1, 0)$, decreasing on $(0,1]$, $c_2(\beta)<0< c_1(\beta)$
and $c_2(\beta)\leq c_3(\beta)$, and for all $x\in (-1,0]$
we have that $f(\beta, x)>x$;
\item there exists constants $A^*_1, A^*_2$ and $\tau\geq 2$ such that for all $\beta$
\[
A^*_1|x|^{\tau-1}\leq |D_x f_{\beta}(x)|\leq A^*_2|x|^{\tau-1}
\]
and
\[
\frac{|D_x f_{\beta}(x)|}{|D_x f_{\beta}(y)|}\leq \exp \bigg(C_*\left|\frac{x}{y}-1\right|\bigg)     
\]
\end{enumerate} 
\end{definition}
\begin{lemma}
Fixed $\tilde{\alpha}\geq 2$ the family $f(\beta, x) := T_{\tilde{\alpha}, \beta}(x)$ is a regular family.
\end{lemma}
\begin{proof}
Item (1), (2) and the first part of item (3) are trivial, 
the second part of item (3) follows from the fact that 
\[
(\alpha-1)\ln\left(\frac{|x|}{|y|}\right)\leq (\alpha-1)\left(\frac{|x|}{|y|}-1\right).
\]
\end{proof}

\begin{definition}
A parameter $\beta$ is called a \textbf{perturbable parameter} if there exists a constant $\varepsilon^*>0$
such that
\begin{enumerate}
\item for every $\delta\in (0, \epsilon^*)$ and $n\geq 1$, if $x\in I$ satisfies $f^i_{\beta}(x)\notin (-\delta, \delta)$
and $f^n_{\beta}(x)\in (-\delta, \delta)$ then $|(f_{\beta}^n)'(x)|\geq \epsilon^*$, 
\item for all $n\geq 1$, $c_n(\beta)\geq \epsilon^*$ and $f_{\beta}$ has no stable periodic point,
\item $\lim_{n\to+\infty}\partial_{\beta} f_{\beta}^n(c_0(\beta))/\partial_x f_{\beta}^{n-1}(c_1(\beta))= Q^*\neq 0$.
\end{enumerate}
\end{definition}

\begin{lemma}
Fixed $\tilde{\alpha}\geq 2$, if we denote by $f_{\beta}(x):=T_{\tilde{\alpha}, \beta}(x)$, the parameter 
$\beta=1$ is a perturbable parameter.
\end{lemma}
\begin{proof}
Item (2) in the definition of perturbable parameter is trivial, since $c_2(1) = -1$, 
which is a fixed point.

Item (3) follows from the chain rule for the derivative with respect to the parameter, i.e.,
\[
\frac{\partial}{\partial \beta}(f(\beta, g(\beta, x)) = \frac{\partial f}{\partial \beta}(\beta, g(\beta, x))+\frac{\partial f}{\partial x}(\beta, g(\beta, x))\frac{\partial g}{\partial \beta}(\beta, x).     
\]
This allows us to check item (3); by a straightforward computation we have that
\[
\frac{\partial f}{\partial \beta}(1,-1)=2 \quad,  \frac{\partial f}{\partial x}(1,-1)=2\alpha 
\]
and that
\[
(\partial_{\beta} f_{1}^n)(0)=\partial_\beta f_{1}(c_{n-1}(1))+\partial_x f_{1}(c_{n-1}(1))\partial_\beta( f_{1}^{n-1}(0)),    
\]
since $c_{2}(1)=-1$, which is a fixed point, we have that
\[
(\partial_\beta f_{1}^n)(0)=2+2\alpha \partial_\beta(f_{1}^{n-1}(0)),    
\]
which in turn tells us that 
\[
(\partial_\beta f_{1}^n)(0)\sim (2\alpha)^{n-1}
\]
which in turn implies item (3).

The last condition we need to check is condition (1);
we will follow a classical construction from \cite{Ja}. 

We will denote by $\eta$ the positive fixed point of $f_{1}(x)$;
denote by $f_L$ the left branch of $f_{1}(x)$ and by 
$f_R$ the right branch.
We will identify by a string of ``R'' and ``L'' the preimages of 
$\eta$ through $f_R$ and $f_L$, i.e.,
\[
RLLL = f_R^{-1}(f_L^{-1}(f^{-1}_L(f^{-1}_L(\eta)));     
\]
we observe that $L = -\eta$.
We will denote by $L^k$ a sequence of $k$ consecutive ``L'' and similarly for ``R''. 

Outside of the domain $I = (-\eta, \eta)$ the map $f_1(x)$ 
is uniformly expanding.
The preimages $RL^k$ for $k=1, 2, \ldots$ are all bigger than $\eta$, so their left
and right preimages fall in $(-\eta, \eta)$
when taking their left and right preimages $LRL$, $RRL$, $LRLL$, $RRLL$,... 
we obtain a countable partition of $(-\eta, \eta)$.

Denote by $\Delta_k = (RRL^{k+1}, RRL^{k})$ and by $\Delta_{-k} = (LRL^{k+1}, LRL^{k})$;
observe that by construction $\Delta_k$ and $\Delta_{-k}$ are mapped diffeomorphically onto 
$(-\eta, \eta)$ by $f_1^{k+1}$.
Moreover, on $[-1,0) \cup (0,1]$ we have that $f_1$ has negative Schwarzian derivative;
this allows us to show that if $x$ belongs to $\Delta_k$, it will come back to $(-\eta,\eta)$ with derivative bigger 
than $1$, by Koebe distortion lemma.

Now, if $x\notin I$, $f_1^k(x)\notin I$ for $k=1,\ldots n$, $f_1^n(x)\in I$,
since $f_1$ is uniformly expanding outside $I$, the condition is satisfied.

If $x\in I$, then $x$ belongs to some $\Delta_i$ and if we
 denote by $r(x) = |i|+1$ the return time to $I$ then $x$
returns to $I$ after $r(x)$ iterations and $|Df_1^{r(x)}(x)|>1$.

If coming back it enters $(-\delta, \delta)$, then the condition is satisfied;
if it returns to $I\setminus (-\delta, \delta)$, then it will return to $I$ only after
$r(f_1^{r(x)}(x))$ steps, with derivative
\[
|Df_1^{r(f_1^{r(x)}(x))}(x)| = |Df_1^{r(f_1^{r(x)}(x))}(f_1^r(x))Df_1^{r(x)}(x)|>1.     
\]

The only remaining case is when $x$ starts outside $I$ and then hits 
$I\setminus (-\delta, \delta)$ in $k$ steps. In this case the modulus of the derivative 
is bigger than $1$ before $k$ and then we will need at least $r(f_1^k(x))$ steps to get back 
to $I$, guaranteeing that the derivative is bigger than $1$.
\end{proof}

This allows us to use \cite{TheTreYou} to prove the following.
\begin{theorem}[Theorem 1.5 \cite{TheTreYou}]\label{thm:ourthetreyou}
Let $\tilde{\alpha}\geq 2$ and let $f_{\beta}(x) = T_{\tilde{\alpha}, \beta}(x)$;
let $\beta=1$; there exists positive constants $C,\gamma, \lambda, \epsilon$ such that $1$ 
is a density point for the set of parameters $\Omega$ such that 
\begin{enumerate}
\item $f_{\beta}$ has no stable periodic point,
\item for all $n\geq 1$, $|f^n_{\beta}(0)|> \epsilon \exp(-n\gamma)$,
\item for all $n\geq 0$, $|(f_{\beta}^n)'(f_{\beta}(0))|> C \exp(n\gamma)$,
\item for all $n\geq 1$, if $x\in [-1,1]$ satisfies $f_{\beta}^k(x)\neq 0$ for all
$k=1,\ldots n-1$ and $f^n_{\beta}(x)=0$, then $|(f^n_{\beta})'(x)|\geq C\exp(n\lambda)$.
\end{enumerate}
This implies that $\beta=1$ is a density point for the set of parameters that admit an 
absolutely continuous invariant measure and with positive Lyapunov exponent with respect to this measure.
\end{theorem}
\begin{remark}
As pointed out by one of the referees, the family $T_{\alpha, \beta}$ is not $C^3$ for $2<\alpha<3$, 
so many results as in \cite{GaoShen, Shen} do not apply in this interval of exponents.
The results of \cite{TheTreYou} works under lower regularity conditions.

Indeed, many of the technical details in the next sections 
are needed to apply our theory to the maps $T_{\alpha,\beta}$
for $\alpha\in (2, 3)$. The treatment is simplified for 
systems with higher regularity.
\end{remark}

\subsection{Stochastic stability}\label{subsec:stochastic}
We remember that the noise is distributed uniformly, i.e., the mother noise kernel is
\[
\rho(x) = \frac{1}{2}\chi_{[-1,1]}.     
\]

We are interested in answering the following question: if $\mu_0$ 
is the invariant measure for the deterministic system and $\mu_{\xi}$
is the stationary measure for the random dynamical system with noise amplitude
$\xi$, is it true that $\mu_{\xi}$ goes to $\mu_0$ as 
the noise amplitude goes to $0$? And in which sense does this happen, i.e.,
is it convergence in the weak-* topology, or we 
can have stronger statements on the convergence?
This problem is called stochastic stability, and many results have appeared during the years
\cite{AlAr03,AlVil13,ArPaPi,BaVi,Me,Shen}, where stochastic stability is proved under different
hypothesis and regularity assumptions.

In \cite{BaVi} strong stochastic stability is proved that for $C^4$ unimodal maps with nondegenerate 
critical points, negative Schwarzian derivative and such that, if $c$ is the critical point,
there exists $\gamma>0$, $\lambda_c>1$, $H_0\geq 1$, $e^{2\gamma}<\sqrt{\lambda_c}$ such that 
\begin{itemize}
     \item $|T^k(c)|\geq e^{-\gamma k}$ for all $k>H_0$
     \item $|(T^k)'(T(c))|\geq \lambda_c^k$ for all $k>H_0$
     \item $f$ is topologically mixing on the interval bounded by $c_1$ and $c_2$.
\end{itemize}
This means that  
if $f_{\xi}$ is the density of 
$\mu_{\xi}$ and $f_0$ is the invariant density of the a.c.i.p. of $T$, we have that
$f_{\xi}$ converges to $f_{0}$ in $L^1$ norm.

The argument goes as follows, the condition above allows the authors in \cite{BaVi}
to construct a uniformly expanding\footnote{with respect to an adapted Riemann metric 
by conjugating the Perron-Frobenius operator by multiplication with a cocycle} 
tower extension of the dynamic $\hat{T}:\hat{I}\to \hat{I}$, where
$\hat{I}\subset \mathbb{N}\times [-1,1]$ is the union of sets of the form
$\{k\}\times B_k$ and the $B_k$'s are a partition of full measure of $[-1,1]$.
If $\Pi(k, x) = x$ is the projection taking a point in $\{k\}\times B_k$, we have
that $\Pi\circ \hat{T}= T\circ \Pi$.

This tower construction, as constructed in \cite{BaVi} works also for all 
deterministic perturbations $T(x)+\omega$ where $\omega<\epsilon_0$,
so they are able to construct an extension of the random dynamical system
$\hat{T}_{\xi}:\hat{I}\to \hat{I}$ such that $\Pi\circ \hat{T}_{\xi}= T_\xi\circ \Pi$,
and using a perturbation argument prove the following theorem.

\begin{theorem}\label{thm:sstability}
There exists an $\xi_0>0$ such that for all $\xi\in [0, \xi_0)$
the random dynamical system $\hat{T}_{\xi}$ on $\hat{I}$ admits a unique stationary measure
with density $\hat{f}_{\xi}$ in $BV$ with respect to the Lebesgue measure $\hat{m}$ in $\hat{I}$.
Moreover 
\[
\lim_{\xi\to 0^+}||\hat{f}_{\xi}-\hat{f}_0||_{BV}\to 0,
\]
which implies that $f_{\xi}:=\Pi_* \hat{f}_{\xi}$ converges to $f_0:=\Pi_* \hat{f}_{0}$ in $L^1$,
$\mu_{\xi}=f_{\xi}dm$ is a stationary measure for $T_{\xi}$, $\mu_0=f_0 dm$ is an invariant measure
for $T$ and, there exists $\tilde{C}>0, 0<\tilde{\theta}<1$ such that for all $\xi\in [0, \xi_0)$ 
\[
||P^n_{\xi}|_{\mathcal{U}_0}||_{BV}< \tilde{C}\tilde{\theta}^n.     
\]
\end{theorem}
We will not give a full proof of the Theorem, since it is quite a 
technical argument and the estimates can be done verbatim, but we will 
show where we can relax the hypothesis of negative Schwarzian derivative 
on the whole domain and the hypothesis that the map is $C^4$ on the whole domain.
\begin{proof}[Sketch of proof]
In the following, let $f(x) = T_{\alpha, \beta}(x)$, where $\beta$
is a stochastic parameter obtained from Theorem \ref{thm:ourthetreyou}.
Without loss of generality, to avoid cluttering with constants, we assume the parameter satisfies
\begin{itemize}
     \item $|c_n(\beta)|>e^{-\gamma n}$, for some small $\alpha$, $n\geq 1$ (slow recurrence to the critical point),
     \item $|(f^{n})'(c_1(\beta))|>\lambda_c^n$, for some $\lambda_c>1$, for all $n\geq 1$ (expansivity along the critical orbit),
\end{itemize}

Following \cite{Me} pag. 287 we fix $\lambda>1$ and $\rho<e^{\gamma}$ such 
that 
\[
e^{\gamma}\lambda\rho\leq \lambda_c^{1/\alpha},     
\]
where $\alpha$ is the exponent of $T_{\alpha,\beta}$, and letting $\gamma<\beta_1<\beta_2<2 \gamma$, 
the condition above implies that
\[
e^{\beta_i/2}\lambda\rho\leq \lambda_c^{1/\alpha}
\]
for $i=1,2$.

Let $c_k=f^k(0)$, and for all $k>0$ let $B_k=[a_{k}, b_k]$ be a set such that $[c_k-e^{-\beta_2 k}, c_k+e^{-\beta_2 k}]\supseteq B_k \supseteq [c_k-e^{-\beta_1 k}, c_k+e^{-\beta_1 k}]$;
due to the slow recurrence to the critical point, we have that $0 \notin B_k$ for all $k>0$;
let $B_0=[-1,1]$.

We fix a small $\delta>0$, to guarantee that once in the neighborhood $(-\delta, \delta)$ we will 
go up enough levels of the tower and, denoting by $f_t(x) = f(x)+t$ and, letting $E_k = B_k\times \{k\}$ for $k\geq 0$ and 
$\hat{I} = \bigcup_{k\geq 0} E_k$; we define $\hat{f}_t:\hat{I}\to \hat{I}$ 
\[
\hat{f}_t(x,k)=\left\{\begin{array}{cc}
(f_t(x), k+1) & \textrm{if $k\geq 1$ and $f_t(x)\in B_{k+1}$}\\
(f_t(x), 1) & \textrm{if $k=0$ and $x\in (-\delta, \delta)$}\\
(f_t(x), 0) & \textrm{otherwise}.
\end{array}\right.
\]
for all $t\in (-\epsilon_0, \epsilon_0)$ (with $\epsilon_0$ small).
The $\Pi$ map is defined as $\Pi(x,k)=x$.

We define the unperturbed cocycle $\omega_0:\hat{I}\to \mathbb{R}$
\[
\omega_0(x, k)=\left\{
\begin{array}{cc}
\frac{\lambda^k}{|(f^k)'(f^{-k}_+(x,k))|} & \textrm{if $(x,k)\in Im(\hat{f}^k)$}\\
0 & \textrm{otherwise}
\end{array}  
\right.
\]
where $f^{-k}_+(x,k) = y $ is the unique point in $(0, \delta)$ such that $\hat{f}^k((y,0))=(x,k)$,
we will not define the perturbed cocycle $\omega_{\epsilon}$ since the negative 
Schwarzian derivative hypothesis enters into play only in the proof of 
the properties of $\omega_0$

The negative Schwarzian derivative hypothesis is used only in 
\cite{BaVi} Lemma 4 and the Sublemma in Section 4.

We start by showing how to adapt the proof of \cite[Lemma 4]{BaVi}.
Note that the support of the cocycle $\omega_0$ in $E_k$ is an interval for each $k\geq 1$,
with endpoints in the set $\partial E_k \cup \{\hat{f}^k(0,0),\hat{f}^k(\delta,0), \hat{f}^k(-\delta,0)\}$.

For $k\geq 1$ let the subintervals of $E_k$ defined as $\beta^+_k = \{(y,k )\mid f(y)>b_{k+1}-\epsilon\}$
and $\beta^-_k = \{(y,k )\mid f(y)<a_{k+1}+\epsilon\}$, and by 
$\gamma^+_k$ and $\gamma^-_k$ respectively their intersection 
with the set $\{\omega_{\epsilon} (x,k)\neq 0\}$.

For $(y,k)\in \gamma^+_k$, and similarly for $\gamma^-_k$ we have that 
\[
\frac{\omega_0(y, k)}{|f'(y)|}=\frac{\lambda^k}{|(f^{k+1})'(\hat{f}_+^{-k}(y,k))|},     
\]
remark that neither $\hat{f}_+^{-k}(\gamma^+_k)$
nor $\Pi(E_j\cap \textrm{supp}(\omega_0))$ contain $0$ (refer to the proof \cite[Lemma 4]{BaVi}, Line 5), 
so the Schwarzian derivative of $f|_{\hat{f}_+^{-k}(\gamma^+_k)}$ is defined and negative 
and similarly for all its iterates. 
Therefore $|(f^{k+1})'(\hat{f}_+^{-k}(y,k))|$ has a unique maximum and 
Lemma 4 follows under our weaker hypothesis, by exchanging the order of 
the arguments in Line 4 and Line 5.

A similar argument works for the Sublemma in \cite[Section 4]{BaVi}, above equation 4.3, 
since $f^n_t$ has no critical points in $\underline{\gamma}$, 
and the point where the Schwarzian derivative is not defined correspond to the critical points,
the Schwarzian derivative of $f^n_t$ is defined and negative, implying that 
the function $g^{(n)}$ has at most a local minimum on $\underline{\gamma}$. 

We need to assess the lack of full $C^4$ regularity; the only place 
where the $C^4$ regularity of the map is used is in \cite[``Climbing the tower'' pag. 497]{BaVi}, 
to prove the regularity of the function $K(x)$ defined as 
\[
K(x)=\frac{|f'(x_{-})|}{|f'(x)|},     
\]
where $x_{-}$ is the unique point with $x_{-}\neq x$ and $f(x)=f(x_{-})$.
We need to prove that there exists finite constants $K$ and $\tilde{K}$ such that
\[
\sup_{x\neq 0} K(x)\leq K, \quad \Var_{x\neq 0}(K(x))\leq \tilde{K}.
\]
Remark that in our family, we have that
$K(x)\equiv 1$ for all $x\neq 0$; so these are trivially satisfied.

The proof then follows directly from the estimates in \cite{BaVi}.
\end{proof}

\begin{remark}
While we fixed the uniform noise kernel, the class of noise kernels for which the result in \cite{BaVi}
holds is larger: in our framework of rescaled noise $\rho_{\xi}(x) =\rho(x/\xi)/\xi$ they can be restated 
as the fact that $\rho$ is bounded (which follows from Bounded Variation) and the fact that,
if we denote by $J = \{t\mid \rho(t)>0\}$, $0\in J$ and $\ln(\rho|_J)$ is concave.  
\end{remark}
This has an important consequence, i.e., continuity of the Lyapunov exponent near $0$.
\begin{corollary}
In the hypothesis of Theorem \ref{thm:sstability}, letting $T(x)=T_{\alpha, \beta}(x)$
\[
\lim_{\xi\to 0^+} \int_{-1}^1 \ln(|T'|)f_{\xi}dm = \int_{-1}^1 \ln(|T'|)f_0 dm.     
\]
\end{corollary}
\begin{proof}
By direct computation
\begin{align*}
&\left|\int_{-1}^1 \ln(|T'|)(f_{\xi}-f_0) dm\right|=\left| \int_{-1}^1 \ln(|T'|)\Pi_*(f_{\xi}-f_0) dm\right| \\
&= \left|\sum_{k\in \mathbb{N}}\int_{\{k\}\times B_k}\ln(|T'(\Pi(\hat{x})|)(\hat{f}_{\xi}-\hat{f}_0) d\hat{m}(x)\right|\leq ||\ln(|T'|)||_{L^1([-1,1])}||\hat{f}_{\xi}-\hat{f}_0||_{BV(\hat{I})},
\end{align*}
where $\hat{m}$ is the Lebesgue measure on $\hat{I}$,
which implies the thesis since $\hat{f}_{\xi}$ converges to $\hat{f}_{0}$ in $BV(\hat{I})$. 
\end{proof}

\begin{corollary}[Corollary of \cite{BaVi}]\label{cor:bavi}
Let $T_{\alpha,\beta}$; fix $\alpha\geq 2$ and let the mother kernel be $\rho(x) = \chi_{[-1,1]}$, i.e, 
the noise in our random dynamical system is the uniform noise.
Then $\beta=1$ is a density point for the set of parameters
$\Omega$ for which there exists a $\xi_0>0$ such that:
\begin{enumerate}
\item for all $\xi\in [0,\xi_0)$ there exists a unique stationary measure $\mu_{\xi}$,
\item the density of the stationary measure $\mu_{\xi}$ converges to the density of the deterministic system in $L^1$ as $\xi$ goes to $0$ 
(strong stochastic stability),
\item $\int_{-1}^1\ln(|T'_{\alpha,\beta}|)d\mu_{\xi}$ is a continuous function of the noise amplitude in $[0,\xi_0)$,
\item there exists $\tilde{C}>0, 0<\tilde{\theta}<1$ such that for all $\xi\in [0, \xi_0)$
\[
||P^n_{\xi}||_{BV}\leq \tilde{C}\tilde{\theta}^n.      
\]
\end{enumerate}
In particular, hypothesis D1, D2, R1, R2 and R4 of Theorem \ref{thm:result} are satisfied.
\end{corollary}
\begin{remark}
All the arguments presented in the sketch of the proof above are already 
known in literature, see \cite{Me}.
\end{remark}

We prove now that hypothesis D3 is also satisfied.
\begin{lemma}\label{lemma:derlp}
For $\alpha\in [2,+\infty)$, $\beta\in (0,1]$ 
\[
\ln(|T'_{\alpha, \beta}|)\in L^p([-1,1]),     
\]
for all $p\geq 1$.
\end{lemma}
\begin{proof}
Follows by computation; let $x<0$, the $x>0$ case is analogous.
\[
T'_{\alpha, \beta}(x) = 2\beta \alpha (-x)^{\alpha-1},    
\]
therefore 
\[
\ln(|T'_{\alpha, \beta}|) = \ln(2)+\ln(\beta)+\ln(\alpha)+(\alpha-1)\ln(|x|),     
\]
which is in $L^p([-1, 1])$ for all $p\geq 1$ since $\ln(|x|)$ is in $L^p([-1,1])$ for all $p\geq 1$.
\end{proof}

We need now to identify under which conditions hypothesis R3 is satisfied.

\subsection{Large noise limit}
By corollary \ref{cor:bignoise} as the amplitude of the noise $\xi$ grows, 
we have that $f_{\xi}$ converges to the uniform density on $[-1,1]$.

Fixed $\beta=1$ we define the following function, the large noise limit of
the Lyapunov exponent of $T_{\alpha,\beta}$:
\[
\Lambda(\alpha) = \int_{-1}^1 \ln(|T'_{\alpha}|)\frac{dm}{2} = \ln(2)+\ln(\alpha)+1-\alpha
\]
This is a decreasing function of $\alpha$, for $\alpha\geq 2$, moreover $\Lambda(2)>0$ and the function
$\Lambda$ has a zero $\tilde{\alpha}$ contained in the interval $[2.67834, 2.67835]$
\footnote{obtained with Julia ValidatedNumerics package}.
A plot of $\Lambda$ is found in figure \ref{fig:exponent}.

\begin{figure}
 \centering
 \includegraphics[width=100mm]{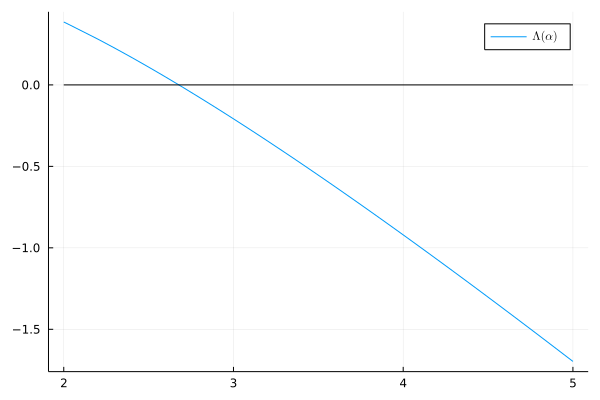}
 % exponents.ps: 0x0 px, 300dpi, 0.00x0.00 cm, bb=
 \caption{The graph of $\Lambda(\alpha)$ \label{fig:exponent}}
\end{figure}

\begin{corollary}
For $\alpha>\tilde{\alpha}$, the map $T_{\alpha, 1}$ presents Noise Induced Order.
\end{corollary}
\begin{proof}
By Corollary \ref{cor:bavi}, hypothesis D1, D2, R1, R2, R4 are satisfied for $T_{\alpha, 1}$.
By Lemma \ref{lemma:derlp} hypothesis D3 is satisfied.
If $\alpha>\tilde{\alpha}$ hypothesis R3 is also satisfied, and by Theorem \ref{thm:result}
we have the thesis.
\end{proof}

\subsection{Behavior as the parameter $\beta$ varies}
In this section we study the behavior of the Lyapunov exponent of $T_{\alpha, \beta}$ in presence of 
noise, when we fix $\alpha$ and vary $\beta$.

We extend the large noise amplitude limit function to allow also $\beta$ to vary;
by a simple computation
\[
\Lambda(\alpha, \beta) := \int_{-1}^1 \ln(|T'_{\alpha, \beta}|)dm = \ln(2)+\ln(\alpha)+1-\alpha+\ln(\beta).
\]
Since $\beta$ belongs to $(0,1]$, this is an increasing function of $\beta$, so, if $\alpha>\tilde{\alpha}$ and 
$\beta<1$ we have that the Lyapunov exponent for big noise sizes is negative.

\begin{lemma}\label{lemma:infdist}
Let $T_{\alpha, \beta}$ be the map defined in equation \eqref{eq:model}.
Then, for $h>0$ we have that
\[
|T_{\alpha, \beta+h}(x)-T_{\alpha, \beta}(x)|\leq 2h|x|^{\alpha}\leq 2h.
\]
\end{lemma}
\begin{proof}
We will prove the inequality on $[0,1]$, the conclusion follows by symmetry:
\[
|2(\beta+h)x^{\alpha}-2\beta x^{\alpha}|\leq 2h|x|^{\alpha}.
\]
\end{proof}
\begin{remark}\label{rem:infdist}
It is possible to compute an estimate also as $\alpha$ varies, proving 
the inequality on $[0,1]$, the conclusion follows by symmetry:
\[
|2(\beta+h)x^{\alpha+k}-2\beta x^{\alpha}|\leq |2(\beta+h)x^{\alpha+k}-2\beta x^{\alpha+k}|
+|2\beta x^{\alpha+k}-2\beta x^{\alpha}|;
\]
we focus now on
\[
|2\beta x^{\alpha+k}-2\beta x^{\alpha}| =k 2\beta|x|^{\alpha}\frac{|x|^k-1}{k};
\]
the inequality is written in this specific form
\[
k |x|^{\alpha}\frac{|x|^k-1}{k} = k |x|^{\alpha}\ln|x|+O(k^2),
\]
therefore, for small $k$ we have that, for some constant $C$
\[
||T_{\alpha+k,\beta+h}-T_{\alpha,\beta}||_{\infty}\leq |h|+C|k|,
\]
since $|x|^{\alpha}\ln(|x|)$ goes to $0$ for $x\to 0$ and has bounded derivative in $[0,1]$ for any $\alpha> 1$.
\end{remark}

\begin{definition}
Fix $\alpha\geq 2$. In the following we will denote by $L_{\xi, \beta}$ 
the annealed transfer operator of $T_{\alpha, \beta}$ with noise amplitude $\xi$, 
$P_{\xi, \beta}$ the associated 
annealed Perron-Frobenius operator.
If a unique stationary measure exists, we will denote it by $\mu_{\xi, \beta}$
and its density by $f_{\xi,\beta}$.
\end{definition}

\begin{corollary}\label{cor:changebeta}
Suppose there exist $\beta, \xi$ such that $P_{\xi, \beta}$ contracts the space of 
average $0$ function in $L^1$; then there exists an $\epsilon>0$ such that for all 
$0<h<\epsilon$ the operator $P_{\xi, \beta+h}$ contracts the space of average $0$ 
functions in $L^1$.
\end{corollary}
\begin{proof}
This follows from Lemma \ref{lem:contmap} and Lemma \ref{lemma:infdist};
if $C, \theta$ are the contraction constants of $P_{\xi, \beta}$ then 
\begin{align*}
||P^n_{\xi,\beta+h}|_{\mathcal{U}_0}||_{L^1\to L^1}&\leq ||P^n_{\xi,\beta}|_{\mathcal{U}_0}||_{L^1\to L^1}+||P^n_{\xi,\beta+h}-P^n_{\xi,\beta}||_{L^1\to L^1}\\
&\leq C\theta^n+3 \frac{C}{1-\theta} h ||\rho_{\xi}||_{BV}.
\end{align*}
If $N$ is such that $C\theta^N<1$ and $h$ is small enough, this implies that 
\[||P^n_{\xi,\beta+h}|_{\mathcal{U}_0}||_{L^1\to L^1}<1.\] 
\end{proof}

\begin{corollary}\label{cor:distbeta}
Suppose there exist $\beta, \xi$ such that $P_{\xi, \beta}$ contracts the space of 
average $0$ function in $L^1$ with constants $C,\theta$;  then if $h$ is small enough
\[
||f_{\xi,\beta+h}-f_{\xi,\beta}||_{L^1}\leq 6 h ||\rho_{\xi}||_{BV} \frac{C}{1-\theta}
\]
\end{corollary}
\begin{proof}
From Corollary \ref{cor:changebeta}, we get that $P_{\xi,\beta+h}$ contracts the space
of average $0$ functions in $L^1$, therefore there exists a unique stationary density
for $P_{\xi,\beta+h}$.

Let $N$ such that $C\theta^N<1/2$, then
\begin{align*}
||f_{\xi, \beta+h}-f_{\xi, \beta}||_{L^1}&\leq ||P_{\xi, \beta}^n(f_{\xi, \beta+h}-f_{\xi, \beta})||_{L^1}+||(P_{\xi, \beta}^n-P_{\xi, \beta+k}^n) f_{\xi, \beta+h}||_{L^1}\\
&\leq \frac{1}{2}||f_{\xi, \beta+h}-f_{\xi, \beta}||_{L^1}+||(P_{\xi, \beta}^n-P_{\xi, \beta+k}^n) f_{\xi, \beta+h}||_{L^1}
\end{align*}
as in the proof of Lemma \ref{lemma:continuity}, and therefore, since $||f_{\xi, \beta+h}||_{L^1}=1$ we have
\[
||f_{\xi, \beta+h}-f_{\xi, \beta}||_{L^1}\leq 6 h ||\rho_{\xi}||_{BV} \frac{C}{1-\theta}.
\]
\end{proof}

\begin{lemma}\label{lemma:Holder}
Fix $\alpha\geq 2$ and let
\[
\lambda_{\xi}(\beta):=\int_{-1}^1 \ln(|T'_{\alpha,\beta}|)f_{\xi, \beta} dm.
\]
If $P_{\xi_0, \beta_0}$ contracts the space of average $0$ functions in $L^1$, then 
the function $\lambda_{\xi_0}$ is defined at $\beta_0$ and it is H\"older continuous with respect to $\beta$ at  
$\beta_0$.
\end{lemma}
\begin{proof}
We observe now that, by H\"older inequality
\[
||f_{\xi,\beta+h}-f_{\xi,\beta}||_{L^r}\leq (||f_{\xi,\beta+h}-f_{\xi,\beta}||_{L^1})^{1/r}(||f_{\xi,\beta+h}-f_{\xi,\beta}||_{L^{\infty}})^{1-1/r},
\]
and that
\[
||f_{\xi,\beta+h}-f_{\xi,\beta}||_{L^{\infty}}\leq 6 ||\rho_{\xi}||_{BV}
\]
Since $\ln|x|$ is in $L^p([-1,1])$ for $p>1$, the result follows.
\end{proof}
\begin{remark}
Under stronger hypothesis on the noise kernel it is possible to prove further regularity 
results on 
\[
\lambda(\xi, \alpha, \beta) = \int_{-1}^1 \ln(|T'_{\alpha,\beta}|)d\mu_{\xi, \alpha, \beta},
\] 
where the function above is defined if there is a unique stationary measure 
for $\mu_{\xi, \alpha, \beta}$ for the annealed transfer operator of $T_{\alpha, \beta}$,
using the linear response theory for random dynamical systems developed in \cite{GaGiu, GaSe}.
\end{remark}

\begin{corollary}
Fixed $\alpha\geq\tilde{\alpha}$ there exists an $\epsilon(\alpha)$ such that
for all $\beta\in (1-\epsilon(\alpha), 1]$ the map $T_{\alpha,\beta}$ 
presents Noise Induced Order.
\end{corollary}
\begin{proof}
Let $\beta=1$. This is the full branch case, by the results in Subsection \ref{subsec:stochastic}
we know that there exists an interval $[0,\xi_0)$ and $C>0,0<\theta<1$ 
such that for all $\xi\in [0,\xi_0)$, we have that 
$||P^n_{\xi,1}|_{\mathcal{U}_0}||\leq C\theta^n$.

Therefore the stationary measure 
$\mu_{\xi, 1}$ is unique for all $\xi \in [0,\xi_0)$ and by Subsection \ref{subsec:stochastic}
there is a $\xi_1$ such that for all $\xi \in [0,\xi_1]$ we have that
\[
\int_{-1}^1 \ln(|T'_{\alpha, 1})d\mu_{\xi, 1}>0.
\]

Let $\hat{\xi} = \min(\xi_0, \xi_1)$; fix a $\xi\in [0,\hat{\xi})$, and let $\epsilon_0$
such that for all $h<\epsilon_0$ the operator $P_{\xi,1-h}$ contracts the space of average
$0$ functions in $L^1$; this $\epsilon_0$ exists by Corollary \ref{cor:changebeta} and depends on $\alpha$.
Then by Corollary \ref{cor:distbeta}, we have that for all $h \in [0, \epsilon_0)$  
\begin{align*}
&\left|\int_{-1}^1 \ln(|T'_{\alpha, 1}|)f_{\xi, 1}dm-\int_{-1}^1 \ln(|T'_{\alpha, 1-h}|)f_{\xi, 1-h}dm\right|\\ 
\leq& \left|\int_{-1}^1 \ln(|T'_{\alpha, 1}|)f_{\xi, 1}dm-\int_{-1}^1 \ln(|T'_{\alpha, 1-h}|)f_{\xi, 1}dm\right|+\\
&\left|\int_{-1}^1 \ln(|T'_{\alpha, 1-h}|)f_{\xi, 1}dm-\int_{-1}^1 \ln(|T'_{\alpha, 1-h}|)f_{\xi, 1-h}dm\right|\\
\leq& \ln(1-h)+2 h ||\ln(|T'_{\alpha, 1-h}|)||_{L^1} 3||\rho_{\xi}||_{BV} \frac{C}{1-\theta}. 
\end{align*}

Therefore, for the $\xi$ fixed above there exists an $\epsilon_1<\epsilon_0$, depending on $\alpha$, 
such that for all $h\in [0, \epsilon_1)$
\[
\int_{-1}^1 \ln(|T'_{\alpha, 1-h}|)f_{\xi, 1-h}dm>0.
\] 
 
Recall now that the big noise amplitude limit of the Lyapunov exponent for $T_{\alpha, 1-h}$ 
is given by 
\[
\int_{-1}^1 \ln(|T'_{\alpha, 1-h}|)\frac{dm}{2}=\ln(2)+\ln(\alpha)+1-\alpha+\ln(1-h);
\]
therefore, if $\alpha>\tilde{\alpha}$, there exists an $\epsilon_2$ (depending on $\alpha$) 
such that for all $h\in [0,\epsilon_2)$ the big noise amplitude limit is negative.

Let $\epsilon = \min(\epsilon_1, \epsilon_2)$ then, for all $\beta\in (1-\epsilon, 1]$ 
the Lyapunov exponent at noise amplitude $\xi$ is positive and the big noise amplitude limit 
is negative, therefore, we have Noise Induced Order.
\end{proof}

\subsection*{Data availability.}
The script used to produce the data, the data, and a Jupyter notebook 
used for the plot in Figure \ref{fig:Lyap5} can be found at 
\url{https://github.com/orkolorko/UnimodalNIO}.

\subsection*{Acknowledgements.}

The author thanks Y. Sato, M. Benedicks, M. Monge and S. Galatolo
for introducing the problem, pointing in the right direction and giving many of 
the tools. The author thanks warmly E. Ghys and A. Blumenthal for reading the paper, posing questions and providing ideas.

The author thanks the anonymous referees, whose questions and comments led  
to a rewrite of the article, which we hope is clearer.

This paper is dedicated to W. Tucker in occasion of his 50th birthday.

\subsection*{Compliance with Ethical Standards}
The author thanks the ICTP for the hospitality and was partially supported by CNPq, University of
Uppsala and KAW grant 2013.0315. UFRJ, CAPES (through the programs PROEX and the CAPES-STINT project "Contemporary topics in non uniformly hyperbolic
dynamics").

The author is currently under ``Afastamento do país para qualificação profissional, apresentação de trabalhos técnico-científicos e colaboração institucional do pessoal docente e técnico-administrativo'' 
from UFRJ and is currently a Specially Appointed Associate Professor at Hokkaido University.

The author would like to thank Prof. Hiroki Sumi and Kyoto university for their hospitality 
during the final revision of this article.

The author has no competing interests to declare that are relevant to the content of this article.

\bibliographystyle{abbrv}
\bibliography{UnimodalNIO}

\end{document}